\title{A Driven Tagged Particle in Asymmetric Exclusion Processes}
\author{Zhe Wang\thanks{
		\'{E}cole Polytechnique F\'{e}d\'{e}rale de Lausanne, D\'{e}partement de Math\'{e}matiques, 1015 Lausanne,
		Switzerland
}        
}
\date{Jul 31, 2019}

\documentclass[11pt,a4paper]{article}

\usepackage[english]{babel}
\usepackage[margin=1.0 in]{geometry}
\usepackage[latin1]{inputenc}

\usepackage{amsmath}
\usepackage{amsfonts}
\usepackage{amssymb}
\usepackage{fourier}
\usepackage{graphicx}
\usepackage{tikz}
\usepackage[inline]{enumitem}
\usepackage{hyperref}

\newtheorem{theorem}{Theorem}[section]
\newtheorem{corollary}{Corollary}[section]
\newtheorem{lemma}{Lemma}[section]

\newtheorem{remark}{Remark}
  
\newtheorem{definition}{Definition}[section]
\numberwithin{equation}{section}

\newenvironment{proof}{{\sc Proof}:}{~\hfill $\square$}

\def\mathbi#1{\textbf{\em #1}}
\def\abs#1{\left\vert #1 \right\vert}

\begin{document}
	\newpage
	\maketitle
	\begin{abstract}
		We consider the asymmetric exclusion process with a driven tagged particle on $\mathbb{Z}$ which has different jump rates \textcolor{black}{from other particles} and show that the tagged particle can have a ballistic behavior when \textcolor{black}{the non-tagged particles have non-nearest-neighbor jump rates}. We show \textcolor{black}{the existence of }some non-trivial invariant measures for the environment process viewed from the tagged particle. \textcolor{black}{Our arguments are based on coupling, the martingale approach, and analyzing currents through fixed bonds.}
	\end{abstract}
	
	\section{Introduction}
	The \textcolor{black}{exclusion process} on the lattice $\mathbb{Z}^d$ with a driven tagged particle can be formally described as: a collection of red particles and a tagged green particle \textcolor{black}{performing continuous-time} random walks on the lattice $\mathbb{Z}^d$ \textcolor{black}{with respect to} the exclusion rule, \textcolor{black}{i.e. at most one particle is at each site and jumps are suppressed if the target site is already occupied}. Red particles have independent exponential clocks with rates $\lambda=\sum_z p(z)$. When a clock rings, the particle at site $x$ jumps to a vacant site $x+z$ with probability $\frac{p(\cdot)}{\lambda}$; the jump is suppressed if the site $x+z$ is occupied. The green tagged particle follows similar rules, but it has different jump rates $q(\cdot)$. We would like to study the long-time behavior of the displacement $D_t$ of the tagged particle.
	
	\textcolor{black}{
	The behavior of the tagged particle is mostly studied when $p(\cdot) = q(\cdot)$. Limit theorems for the displacement $D_t$ were obtained by works  \cite{Ar, Sa, Ki, KV,Va, SVY}. The environment process $\xi_t$ viewed from the tagged particle has a class of invariant and ergodic measures: Bernouli measures $\mu_\rho$ with parameters $\rho$ ($0\leq\rho\leq 1$). As a consequence, the speed of the tagged particle can be computed explicitly as $(1-\rho)\sum_{z}z\cdot q(z).$ For details, see Chapter III.4 \cite{Li}. The fluctuation of $D_t$ in equilibrium is known to be subdiffusive when $d=1$ and $p(\cdot)=q(\cdot)$ are nearest-neighbor symmetric \cite{Ar}, and diffusive in most other finite-range cases \cite{Ki,KV,Va,SVY}. A powerful method, by Kipnis and Varadhan \cite{KV}, is to study the additive functionals of reversible Markov processes. It is also extended to asymmetric models \cite{Va,SVY}. The only open cases are when non-mean-zero $p(\cdot) = q(\cdot)$ is non-nearest-neighbor in dimension $d= 1$, and when $p(\cdot) = q(\cdot)$ is non-mean-zero in dimension $d=2$.}
	
	\textcolor{black}{
	The case where $d=1$, $p(\cdot)=q(\cdot)$ are nearest-neighbor is special. Particles are trapped and orders are preserved. The gaps between particles follow a zero-range process\textcolor{black}{\cite{Ki}}. The displacement $D_t$ can be considered jointly with the current \textcolor{black}{through} the bond $0$ and $1$ \textcolor{black}{in either the zero-range process\cite{Ki}, or the exclusion process \cite{SV}}. On the other hand, for symmetric jump rates $p(\cdot)$, we can use the stirring system to construct \textcolor{black}{the symmetric exclusion process}, \textcolor{black}{see Chapter VIII.4 \cite{Li85}}. With these considerations, one can study the density fields and apply hydrodynamic limit results to analyze the displacement $D_t$. Some related works are \cite{Ar,JL,SV,FF,Go}.}

	\textcolor{black}{However, when jump rates $p(\cdot),q(\cdot)$ are different and non-nearest-neighbor, the asymptotic behavior of the displacement $D_t$ is less understood. A primary difficulty to provide rigorous proofs is the lack of explicit knowledge of invariant measures for the environment process, which is important in the analysis in \cite{Ar,KV,Va,SVY}. The difference between $p(\cdot)$ and $q(\cdot)$ introduces asymmetry which makes explicit computations of invariant measure seemingly impossible. In dimension $d\leq 2$, it is unclear whether there are multiple invariant measures for different values of density $\rho$, except for two trivial ones, Bernoulli measures $\mu_0$ and $\mu_1$. In $d\geq 3$, Loulakis' result \cite{Lo} provides a partial answer when $p(\cdot)$ is symmetric. Also, in dimension {$d=1$}, the orders of gaps between particles are no longer preserved due to the non-nearest-neighbor assumption on $p(\cdot)$. 
	}\textcolor{black}{Meanwhile, we should notice two special cases in general one-dimensional asymmetric models without a tagged particle: in asymmetric exclusion process, there are stationary blocking measures \cite{BM,FLS}; in the totally asymmetric simple K-exclusion process, invariant measures are also unknown \cite{Se}. In the former model, blocking measures are nontranslation invariant measures concentrated on configurations that are completely occupied by particles on one end, and completely empty to the other end. The existence of blocking measures in principle implies the displacement $D_t$ of a tagged particle may grow sub-linearly, because particles are often blocked by their neighbor particles. For the K-exclusion process, Sepp\"{a}l\"{a}inen \cite{Se} managed to show the hydrodynamic limit of the system with invariant measures unknown. His arguments are based on coupling the process with a growth model.} 
	 
	 \textcolor{black}{Alternatively, we can view the tagged particle in the exclusion process as a special particle driven by an external force, and consider our model as a perturbed system. One approach is to verify the Einstein relation, which connects mobility and diffusivity. The mobility describes the speed of the tagged particle in the perturbed system, while the diffusivity describes the variance of the tagged particle in the unperturbed system. For exclusion processes, the Einstein relation are verified in some symmetric and reversible scenarios \cite{KO,Lo,LOV}}. When dimension $d=1$, $p(\cdot)$ is symmetric and $p(\cdot), q(\cdot)$ are nearest-neighbor, Landim, Olla and Volchan \cite{LOV} \textcolor{black}{showed that the displacement $D_t$ grows as $\sqrt{t}$, and there is an Einstein relation for $D_t$, by studying the dynamics of gaps.} \textcolor{black}{They further} conjectured that $D_t$ grows linearly in $t$ when the mean $\sum_z z\cdot q(z)$ is positive and $p(\cdot)$ is non-nearest-neighbor in $d=1$ or general in $d\geq 2$. This conjecture is \textcolor{black}{partially verified when $d\geq 3$ and $p(\cdot)$ is symmetric \cite{Lo}}, and it remains open for most of the other cases. \textcolor{black}{When $q(\cdot)$ is close to $p(\cdot)$, we can show the displacement $D_t$ grows linearly in $t$ with a corresponding Einstein relation \cite{Lo}. However, the speed of the tagged particle is unknown because there is no explicit formula for the invariant measures. For mixing dynamical environment with spectral gap, Komorowski and Olla \cite{KO} were able to obtain full expansion of invariant measures, and showed the explicit speed and the corresponding Einstein relation.
	}
	
	 \textcolor{black}{Another approach is to study the currents through a fixed bond in the one-dimensional asymmetric exclusion process (AEP) with coupling arguments. The current describes the average number of particles across a site, and it is a natural object especially when $\sum_z z\cdot p(z)$ is non-zero. By couplings, Liggett \cite{Li75,Li77} computed the currents and limiting measure in AEP explicitly for a class of general initial measures. For a more general class of asymmetric conservative particle systems with a blockage, one can show a hydrodynamic limit result when a second type of coupling is available \cite{Ba}. In these systems, the current across the blockage is a key quantity in the hydrodynamic limit because it describes the densities near the blockage. 
	  Although this second type of couplings is different from that in Liggett's \cite{Li75,Li77}, it is available in the one-dimensional AEP case. When jump rates $p(\cdot)$ satisfy certain conditions, Ferrari, Lebowitz and Speer \cite{FLS} showed a coupling of two AEPs and applied this coupling to prove the existence of blocking measures. When a driven tagged particle is present, we can also consider the current across a special site, the (moving) tagged particle, and obtain estimates of currents by coupling different AEPs with a driven tagged particle. }
	 
	 In this article, we will consider the case where $d=1$ and $p(\cdot)$ is non-nearest-neighbor and asymmetric with positive mean $\sum z\cdot p(z) >0$. \textcolor{black}{The main tools are the couplings and martingale arguments. There are two types of couplings similar to those in \cite{FLS,Li75,Li77}. These two types of couplings allow us to compare currents in different processes and obtain estimates of currents. With martingale arguments, we can relate estimates of currents to estimates of the displacement $D_t$ and some invariant measures. In the end, we will show that the displacement $D_t$ grows linearly in $t$ in three scenarios (Theorems \ref{Thm:ballistic behavior for left moving tagged particle}, \ref{Thm:ballistic behavior for fast moving tagged particle}, \ref{Thm:positive speed for slow moving tagged particle}). These results suggest behaviors of the tagged particle depend on jump rates $p(\cdot), q(\cdot)$ and the initial measure in a nontrivial way. By characterizing some nontrivial invariant measure, we will show that the tagged particle can have a positive speed in AEP even when it has a negative drift, i.e.  $\sum_z z\cdot q(z)<0$ (Theorem \ref{Thm:positive speed for slow moving tagged particle}). Some mild assumptions will also be made in the next section.}

\section{Notation and Results\label{chap:Notation and Result}}

In this section, we \textcolor{black}{first introduce the problem, describe the environment process viewed from the tagged particle, next we describe the assumptions, introduce some notation and lastly we state the main results and provide an outline of the proofs.}

A configuration $\xi(\cdot)$ on $\mathbb{Z}\setminus\{0\}$ indicates which sites are occupied relative to the tagged particle:  $\xi(x) = 1$ if site $x$ is occupied, and $\xi(x) = 0$ otherwise. The collection of all configurations $\mathbb{X}=\{0,1\}^{\mathbb{Z}\setminus\{0\}}$ forms  a state space for the environment process $\xi_t$. 

Local functions are \textcolor{black}{functions defined on $\mathbb{X}$ which depend on finitely many $\xi(x)$}. We will use $\mathbi{C}$ to denote the space of local functions on ${\mathbb{Z}\setminus\{0\}}$ and $\mathbf{M}_1$ to denote the space of probability measures on $\mathbb{X}$. \textcolor{black}{Examples of local functions are}: 
\begin{align}
\xi_x(\xi) =& \xi(x),\\
\xi_A(\xi) =& \prod_{x\in A} \xi(x), \text{ A is a finite set of } \mathbb{Z}.
\end{align}

\textcolor{black}{The process $\xi_t$ starting from any initial configuration in $\mathbb{X}$ is a Markov process with generator $\mathit{L}=\mathit{L}^{ex}+\mathit{L}^{sh}$. The action of} $\mathit{L}$ on any local function $f$ is given by:
\begin{align}
\mathit{L}f(\xi) =& (\mathit{L}^{ex}+\mathit{L}^{sh}) f(\xi) \notag \\
=&\sum_{x,y\neq 0}p(y-x)\xi_x\left( 1-\xi_y\right)\left(f(\xi^{x,y})-f(\xi)\right) \notag \\
&+\sum_{z}q(z)\left( 1-\xi_z\right)\left(f(\theta_z\xi)-f(\xi)\right)  \label{Eq:generator}
\end{align}
where $\xi^{x,y}$ represents the configuration after exchanging particles at site $x$ and $y$ of $\xi$,	
\begin{equation} \label{usual exchange}
\xi^{x,y}(z) =
\begin{cases}
\xi(z) & \text{if }z \neq x,y \\
\xi(y) & \text{if }z = x \\
\xi(x) & \text{if }z = y 
\end{cases}
\end{equation}
and $\theta_z\xi$ represents the configuration shifted by $-z$ unit due to the jump of the tagged particle to an empty site at $z$,
\begin{equation} \label{usual translation}
(\theta_z\xi)(x) =
\begin{cases}
\xi(x+z) & \text{if }x \neq -z \\
\xi(z) & \text{if }x = -z 
\end{cases}.
\end{equation} 
\textcolor{black}{The generator $\mathit{L}^{ex}$ corresponds to the motion of red particles, while the generator $\mathit{L}^{sh}$ corresponds to the motion of the tagged particle.}

Denote the probability measure on the space of c\'{a}dl\'{a}g paths on $\mathbb{X}$ starting from a deterministic configuration $\xi_0=\eta$ by  $\mathbb{P}^{\eta,q}$
, and let $\mathbb{P}^{\nu_0,q}=\int\mathbb{P}^{\eta,q}\, d\nu_0(\eta)$ when the initial configuration $\xi_0$ is distributed according to some measure $\nu_0$ on $\mathbb{X}$. We also denote the expectation with respect to $\mathbb{P}^{\nu_0,q}$ 
by $\mathbb{E}^{\nu_0,q}$
. A special initial measure is the step measure $\mu_{1,0}$, which concentrates on the configuration $\xi$, with $\xi(x)=1,$ for $x<0$, and  $\xi(x) = 0$, for $x>0$. 

Lastly, we will use $D_t$ to denote the displacement of the green tagged particle up to time $t$. The main problem is to investigate the long time behavior of $D_t$ when $q(\cdot)$ is different from $p(\cdot)$.

To illustrate the result, we consider the case where red particles have positive drifts, while the tagged particle has jump rates $q(\cdot)$. \textcolor{black}{We want $p(\cdot)$ to} satisfy the following assumptions:

\begin{enumerate}[label=A\arabic{*}]
	\item (Positive) $p(2)= p(-2) > 0,$ and $p(1) > p(-1). $ \label{Asmp:positive} 
	
	\item (\textcolor{black}{Radially Decreasing} and Range 2) $ p(-1) \geq p(-2) $, $p(k)=0$ for all $\abs{k}>2$. \label{Asmp:attractiveness}		
\end{enumerate}
\textcolor{black}{These two assumptions imply $p(1)\geq p(2)$.} It turns out that we can have more general assumptions on the jump rates $p(\cdot)$ \textcolor{black}{and} get similar results:
\begin{enumerate}[label=A'\arabic{*}]
	\item (Positive) $p(k)\geq p(-k)$ for all $k>0$, and  $p(k) > p(-k)$ for some $k$, \label{Asmp:positive2} 
	
	\item (\textcolor{black}{Radially Decreasing}) $ p(x) $ is increasing on negative axis and decreasing on positive axis,  \label{Asmp:attractiveness2}
	
	\item (Finite-range) there is an $R>0$ such that $p(x) = 0$ for $|x|>R$.  \label{Asmp:finite-range2}		
\end{enumerate}
	
\textcolor{black}{Our main results are ballistic behaviors of a driven tagged particle in asymmetric exclusion processes (AEP)}. The first result is the most natural one. When \textcolor{black}{the initial measure is the step measure $\mu_{1,0}$ and the tagged particle has only pure left jump rates}, it has a ballistic behavior towards left, \textcolor{black}{i.e. $\frac{D_t}{t}$ has a strictly negative asymptotic upper bound. } 
\begin{theorem} \label{Thm:ballistic behavior for left moving tagged particle} (Ballistic Behavior of a Tagged Particle in AEP with Only Left Jumps)
	
	Consider the \textcolor{black}{AEP} with a driven tagged particle. Let the jump rates $p(\cdot)$ for the red particles satisfy \ref{Asmp:positive} and \ref{Asmp:attractiveness}, and the jump rates $q(\cdot)$ be supported on negative axis. 
	Then, starting from \textcolor{black}{the step initial measure $\mu_{1,0}$}, we have
	\[ \limsup_{t\to \infty} \frac{D_t}{t} \leq c <0 ,\  \textcolor{black}{\mathbb{P}^{\mu_{1,0},q}}\--a.s. \]
\end{theorem}

When the tagged particle can jump in both directions, we can also obtain ballistic behaviors. \textcolor{black}{If the tagged particle has a drift $\sum_z z\cdot q(z)$ greater than $\sum_z z\cdot p(z)$, its mean displacement is expected to have the same asymptotic lower bound. The second result confirms that under some conditions on $p(\cdot)$ and $q(\cdot)$, the displacement $D_t$ has an asymptotic lower bound $ (1-\rho) \sum_z z\cdot p(z)$. Recall ${(1-\rho) \sum_z z\cdot p(z)}$ is the speed of a tagged particle with $q(\cdot) = p(\cdot)$ in AEP starting from a Bernoulli product measure $\mu_\rho$.}
\begin{theorem} \label{Thm:ballistic behavior for fast moving tagged particle} (Ballistic Behavior of a Fast Tagged Particle in AEP)
	
	 	\textcolor{black}{Consider the {AEP} with a driven tagged particle.} Let the jump rates $p(\cdot)$ for the red particles be supported on site $-2,-1,1$, and jump rates $q(\cdot)$ for the tagged particle be supported on site $-1,1,2$. Assume $p(\cdot),q(\cdot)$ satisfy:
	\begin{enumerate}
		\item $p(-2)<p(-1)$, $q(2)<q(1)$
		\item $p(-1)> q(-1)$, $q(1) > p(1)$
		\item $\sum_z z\cdot p(z)$ >0	 	
	\end{enumerate}
	Then, starting from a Bernoulli product measure $\mu_\rho$ with $\rho \in (0,1)$, we have
	\[ \liminf_{t\to \infty} \frac{D_t}{t} \geq (1-\rho) \sum_z z\cdot p(z) ,\  \mathbb{P}^{\mu_\rho,q}\--a.s. \]
\end{theorem}
\textcolor{black}{The first two assumptions imply that we can couple two continuous-time random walks with jump rates $p(\cdot), q(\cdot)$ such that the walk with $q(\cdot)$ always stays on the right of the walk with $p(\cdot)$. The supports of jump rates $p(\cdot),q(c\dot)$ implies red particles do not jump to the right of the tagged particle, and the tagged particle does not jump to the left of red particles.}

\textcolor{black}{The final result is that a slow tagged particle in AEP can follow the general behavior of red particles even if it has jump rates $q(\cdot)$ with a negative mean ${\sum_z z\cdot q(z) <0}$. By slow, we mean that the size of $q(\cdot)$, $\sum_z q(z)$, is sufficiently small relative to $\sum_z z\cdot p(z)$.} 

\begin{theorem} \label{Thm:positive speed for slow moving tagged particle} (Ballistic Behavior of a Slow Tagged Particle in AEP)
	
	Consider the \textcolor{black}{AEP} with a driven tagged particle. Let the jump rates for the red particles satisfy \ref{Asmp:positive} and \ref{Asmp:attractiveness}. Then, there exist nearest-neighbor jump rates $q(\cdot)$ for the tagged particle and an ergodic invariant measure $\nu_e$ for the environment process viewed from the tagged particle such that, under $\mathbb{P}^{\nu_e,q}$, we have
	\begin{enumerate}[label=\alph*.]
		\item the tagged particle has a negative drift: $q(-1)>q(1)>0$,
		\item the tagged particle has a positive speed under $\mathbb{P}^{\nu_e,q}$, that is,
		\[ \lim_{t\to \infty} \frac{D_t}{t} = m >0,\  \mathbb{P}^{\nu_e,q}\--a.s. \]
	\end{enumerate} 
\end{theorem}

In these results, we have shown \textcolor{black}{ballistic behaviors of the tagged particle in AEP. With arguments to be introduced in Section \ref{sec: Invariant Measure and Lower Bound}, these ballistic behaviors (with estimates) imply the existence of some non-trivial invariant measures, measures other than $\mu_0$ or $\mu_1$, for the environment process viewed from the tagged particle.} In the driven tagged particle problem, the invariant measure is in general impossible to compute due to the break of symmetry. The Bernoulli product measure is no longer invariant. In principle, there could be multiple invariant measures, which makes the behaviors of the tagged particle hard to predict. The following are some remarks on the choices \textcolor{black}{of more general} jump rates $p(\cdot)$ and $q(\cdot)$. 

\begin{remark} \label{rm: general jumprates}
	We can have more general $p(\cdot)$ and $q(\cdot)$. If $p(\cdot)$ satisfy assumptions 
	 \ref{Asmp:attractiveness2}, and \ref{Asmp:finite-range2}, we need:
	\begin{enumerate}
		\item in \textcolor{black}{Theorem \ref{Thm:ballistic behavior for left moving tagged particle}, $p(\cdot)$ are under additional assumption \ref{Asmp:positive2}}, and $q(\cdot)$ are \textcolor{black}{only} supported on negative axis;
		\item in \textcolor{black}{Theorem \ref{Thm:ballistic behavior for fast moving tagged particle}, $p(\cdot), q(\cdot)$ satisfy the second assumption, and red particles starting from the left of the tagged particle always remain to the left of the tagged particle, \textcolor{black}{ i.e. $p(-1)\geq q(-1)$, $q(1) \geq p(1)$, and $p(k)=0, q(-k)=0$ for all $k\geq 2$};}
		\item in \textcolor{black}{Theorem \ref{Thm:positive speed for slow moving tagged particle}, $p(\cdot)$ are under additional assumption \ref{Asmp:positive2}, and  $\sum_z q(z)$ is small relative to the mean ${\sum z \cdot p(z)}$ of $p(\cdot)$, ie. $\sum_z q(z) < c' \sum z\cdot p(z)$ for some small enough $c'>0$}.
	\end{enumerate} 
\end{remark}

%
%
%

To get these three results, we use similar ideas. The proofs of \textcolor{black}{Theorem \ref{Thm:ballistic behavior for left moving tagged particle} and Theorem \ref{Thm:positive speed for slow moving tagged particle} are similar and they are in Section \ref{sec:proof of main}. The proof of Theorem \ref{Thm:ballistic behavior for fast moving tagged particle} is in Section \ref{sec:Lower bounds for a fast tagged particle}}. We will mainly discuss the approach to {Theorem \ref{Thm:positive speed for slow moving tagged particle}}. It consists of three parts.

We first \textcolor{black}{start from any initial measure $\nu_0$ and} obtain a candidate $\bar{\nu}$ for the invariant measure in Theorem \ref{Thm:positive speed for slow moving tagged particle} and some estimates of the displacement $D_t$. Let $N_t$ be the number of red particles which initially start from the left of the tagged particle and \textcolor{black}{move to the right of the tagged particle by} time t. By standard martingale arguments and an algebraic identity, \textcolor{black}{we can see that, up to an error of $q(1)-q(-1)$, a multiple of $\mathbb{E}^{\nu_0,q}\left[\frac{N_t}{t}\right]$ is a lower bound for $\mathbb{E}^{\nu_0,q}\left[\frac{D_t}{t}\right]$.}This is done in Section \ref{sec: Invariant Measure and Lower Bound}. On the other hand, \textcolor{black}{we will show the speed of the tagged particle is $\mathbb{E}^{\nu_0,q}\left[\frac{D_t}{t}\right]$ if $\nu_0$ is ergodic.} This is done in Section \ref{sec:proof of main}.

\textcolor{black}{Next, we want to} prove a positive lower bound for $\mathbb{E}^{\nu_0,q}\left[\frac{N_t}{t}\right]$ for some $\nu_0$, \textcolor{black}{and} we use two steps. The first step is to obtain an estimate for $\mathbb{E}^{\nu_0,q}\left[N_t\right] - \mathbb{E}^{\nu_0,0}\left[N_t\right]$, which allows us to consider the case where the tagged particle does not move. This estimate indicates that the case when the tagged particle is moving slowly \textcolor{black}{can} be viewed as a pertubation of the case when the tagged particle is fixed. This estimate requires a coupling result, which is the main subject in Section \ref{sec:coupling}. \textcolor{black}{The existence of coupling requires mainly assumptions \ref{Asmp:attractiveness2} and \ref{Asmp:finite-range2}, and we will show it in Appendix \ref{sec: existence of coupling}.} 

The second step is to prove a positive current \textcolor{black}{ ${\mathbb{E}^{\nu_0,0}\left[\frac{N_t}{t}\right]}$} for some initial measure $\nu_0$. When the tagged particle does not move, the environment process evolves as the AEP with a blockage at site $0$. \textcolor{black}{We consider the case where $\nu_0$ is the step measure $\mu_{1,0}$ and prove that the current is strictly positive by contradiction}. The idea is to consider the limiting measure of an invariant measure under translations $\{\tau_x \bar\nu\}$ in the Ces\`{a}ro sense. We will get an estimate for this limiting measure by comparing this process with another process called AEP on a half line with creation and annihilation. The analysis of the latter process requires a second coupling argument, and follows results and ideas of Liggett, \cite{Li75,Li77}. The second step is done in Section \ref{sec: Current Zero} and Section \ref{sec: ASEP on the half line with destruction}.

We end this section with some remarks on the coupling result to be introduced in Section \ref{sec:coupling} and \textcolor{black}{the current through a fixed bond}.

\begin{remark} \label{rm: hydro and baha}
	\ 
	\begin{enumerate}
		\item Ferrari, Lebowitz, and Speer considered a coupling in \cite{FLS}. This is the same as the couplings  in Section \ref{sec:coupling}. We give an alternative construction in Appendix \ref{sec: existence of coupling}. See Lemma 4.2 in \cite{FLS} and Theorem \ref{Thm: main step} in Appendix \ref{sec: existence of coupling}. \textcolor{black}{ The main improvement in this article is the couplings of two environment processes when the tagged particle has jump rates different from $p(\cdot)$.}
		\item  If $p(2) > p(-1)+2p(-2)$, we can obtain a positive lower bound for $\frac{1}{t}\mathbb{E}^{\nu_0,0}\left[N_t\right]$ using couplings in Section \ref{sec:coupling}. However, the proof does not work with the general assumption \ref{Asmp:attractiveness} so we will use arguments in Section \ref{sec: Current Zero} and Section \ref{sec: ASEP on the half line with destruction} instead.
		\item 
		The step initial measure $\mu_{1,0}$ gives us the maximal value for ${\limsup_{t\to\infty} \frac{1}{t}\mathbb{E}^{\nu_0,0}\left[N_t\right]}$. By Theorem 2.2 from \cite{Ba} and the couplings in Section \ref{sec:coupling}, we can get positive lower bounds for $\frac{1}{t}\mathbb{E}^{\nu_0,0}\left[N_t\right]$ for more general initial measures $\nu_0$, such as Bernoulli measures $\mu_{\rho}$ for $\rho$ \textcolor{black}{close to $0$ or $1$}. This is indeed a hydrodynamic limit result for the AEP with a blockage. 
		
		\item \textcolor{black}{The current through a fixed bound in the AEP with blockage is of its independent interest. The current in the usual AEP starting from a step initial measure $\mu_{1,0}$ is computed explicitly by Liggett \cite{Li75,Li77} as ${\frac{1}{4}\sum z\cdot p(z)}$. However, in the case where there is a local perturbation, the size of current is open. Whether the value of the current in the perturbed system is strictly smaller than ${\frac{1}{4}\sum z\cdot p(z)}$ not well-understood, and it is known as the "Slow Bond Problem". Recently, there is a progress in the nearest-neighbor case by Basu, Sidoravicius and Sly \cite{BSS}. In the current article, we will show the lower bond in the perturbed case is strictly positive in some non-nearest-neighbor cases (Theorem \ref{thm:positive current when it is not fully bocked}). }
	\end{enumerate}
\end{remark}

	\section{Invariant Measure and the Lower Bound for the Displacement of a Tagged Particle} \label{sec: Invariant Measure and Lower Bound}
In this section, \textcolor{black}{we will assume that $q(\cdot)$ is nearest-neighbor and $p(2)=p(-2)$. This simplifies the computation, for some generalization, see Remark \ref{rm: general jumprates}.} We construct a candidate invariant measure by using the empirical measures. We also relate the displacement $D_t$ of the tagged particle to the current \textcolor{black}{through} bond $(-1,1)$. Most results in this section are standard martingale arguments.

We start with a tightness result on $\mathbf{M}_1$ with weak topology. Since $\mathbb{X}=\{0,1\}^{\mathbb{Z}^d\setminus\{0\}}$ is equipped with the product topology, it is compact. By Prokhorov's Theorem, $\mathbf{M}_1$ is precompact with the weak topology.

\textcolor{black}{Define the (random) empirical measure $\mu_t$ for process $\xi_t$ and its mean $\nu_t$ by their actions on local functions:
$\langle\mu_t,f\rangle:=\frac{1}{t}\int_0^t f(\xi_s)\,ds$, and  $\langle \nu_t,f\rangle:=\frac{1}{t}\mathbb{E}^{\nu_0}[\int_0^t f(\xi_s)\,ds],\forall f \in \mathbi{C}$ and $t>0$. We also have $\nu_0 = \lim_{t\downarrow 0} \nu_t$.
We can then obtain $\bar{v}$ as the weak limit of a subsequence $\nu_{T_n}$.} It is an invariant measure by Theorem B7 \cite{Li}.

Let $\mathcal{F}_t:=\sigma(\xi_s:s \leq t)$ and let $N_t$ be the net number of the red particles moving from the left of the tagged particle to the right \textcolor{black}{of the tagged particle} up to time t (or the integrated current \textcolor{black}{through} bond (-1,1)). Since the tagged particle has only nearest-neighbor jumps, the jumps of the tagged particle do not \textcolor{black}{change the value of} $N_t$ and $N_t$ is the difference of two numbers: 
\begin{equation}\label{eq: Nt}
N_t:=R_t - L_t = \sum_{s\leq t}\chi_{\left\lbrace\xi_s=\xi^{-1,1}_{s^{-}},\xi_s(1)=1,\xi_s(-1) =0\right\rbrace} -\sum_{s\leq t}\chi_{\left\lbrace\xi_s=\xi^{-1,1}_{s^{-}},\xi_s(1)=0,\xi_s(-1) =1\right\rbrace}
\end{equation} 

Under $\mathbb{P}^{\xi,q}$, $R_t$ has (varying) jump rates $\lambda_1(\xi_t)=p(2)(1-\xi_t(1))\xi_t(-1)$, and $L_t$ has \textcolor{black}{(varying)} jump rates $\lambda_2(\xi_t)=p(-2)(1-\xi_t(-1))\xi_t(1)$. \textcolor{black}{By using $\mathbb{P}^{\nu_0,q}$-martingales (Chapter 6.2 \cite{KLO})}, and uniform integrability, we can obtain the following result:

\begin{lemma}\label{prop:Nt and Measure} Let jump rates $p(\cdot)$ satisfy assumption \ref{Asmp:positive}.	For a sequence of $T_n\uparrow \infty$ \textcolor{black}{such that} ${\bar{\nu} = \lim_{n\to\infty}\nu_{T_n}}$ is invariant. It has an estimate:
	\begin{equation} \label{Easy Estimate1}
	\textcolor{black}{\langle \bar{\nu},C_{-1,1}\rangle}= \lim_{n\to \infty}\mathbb{E}^{\nu_0,q}\left[\frac{N_{T_n}}{T_n}\right],
	\end{equation} \textcolor{black}{where $C_{-1,1}=p(2)\xi_{-1}(1-\xi_1)-p(-2)\xi_1(1-\xi_{-1})$.}
	Furthermore, if there is a $C_0>0$, such that 
	\begin{equation}  \label{est:lower bound for current}
	\liminf_{t\to \infty} \mathbb{E}^{\nu_0,q}\left[\frac{N_t}{t} \right] \geq C_0, 
	\end{equation}
	we also have 
	\begin{align}   
	\langle\bar{\nu},\xi_{-1}-\xi_{1} \rangle&\geq \frac{C_0}{p(2)} >0. \label{As to Easy Estimate: LB}
	\end{align}
\end{lemma}
\begin{proof}
	\textcolor{black} {We write two $\mathbb{P}^{\nu_0,q}$--martingales $M_t = R_t-\int_0^t\lambda_1(\xi_s)\,ds$ and $\tilde{M}_t = L_t-\int_0^t\lambda_2(\xi_s)\,ds$. (See Chapter 6.2 \cite{KLO} for details.) Taking expectation with respect to $\mathbb{P}^{\nu_0,q}$}, we obtain
	\[ \langle\nu_{T_n},p(2)\xi_{-1}(1-\xi_{1}))-p(-2)\xi_1(1-\xi_{-1} ) \rangle= \frac{1}{T_n}\mathbb{E}^{\nu_0,q}[N_{T_n}]\]
	Passing through the weak limit, we get the equation \eqref{Easy Estimate1}. As $L_t$ and $R_t$ are both dominated by a Poisson Process with rate 1, $\{\frac{M_t}{t}\}_{t>1}$ is uniformly integrable. Using $p(2)=p(-2)$, we get  \eqref{As to Easy Estimate: LB} from  \eqref{Easy Estimate1},\eqref{est:lower bound for current}.  
\end{proof}

\textcolor{black}{We can also write the displacement of the tagged particle $D_t$ as the difference of two numbers, $r_t$ and $l_t$, the numbers of right jumps and left jumps of the tagged particle:
\begin{equation}\label{eq: Dt}
D_t:=r_t - l_t = \sum_{s\leq t}\chi_{\left\lbrace\xi_s=\theta_1\xi_{s^{-}}\right\rbrace} -\sum_{s\leq t}\chi_{\left\lbrace\xi_s=\theta_{-1}\xi_{s^{-}}\right\rbrace}.
\end{equation}}
With a similar argument, we see the displacement $D_t$ has \textcolor{black}{a lower bound which is a multiple of $C_0$, up to an error (the difference of $q(-1)$ and $q(1)$):} 
\begin{lemma}\label{prop:displacement and spped of tagged} Let jump rates $p(\cdot)$ satisfy assumption \ref{Asmp:positive} and \ref{Asmp:attractiveness}, and $q(\cdot)$ be nearest-neighbor.
	There is a sequence of $T_n \uparrow \infty$ \textcolor{black}{such that} $\bar{\nu} = \lim_{n\to \infty}\nu_{T_n}$ is invariant, and $D_t$ has an estimate:
	\begin{equation} \label{Easy Estimate2}
	q(1)\langle \bar{\nu},1-\xi_1\rangle-q(-1)\langle\bar{\nu}, 1-\xi_{-1}\rangle= \liminf_{t\to \infty}\mathbb{E}^{\nu_0,q}\left[\frac{D_{t}}{t}\right].
	\end{equation}
	Furthermore, \textcolor{black}{if \eqref{est:lower bound for current} holds with $C_0>0$, then}
	\begin{align}  
	\liminf_{t\to \infty}\mathbb{E}^{\nu_0,q}\left[\frac{D_{t}}{t}\right] &\geq \frac{q(1)}{p(2)}C_0 - (q(-1)-q(1)). \label{est: lowerbdd for displacement}
	\end{align}
	
\end{lemma}
\begin{proof}
	It is almost the same as that of Lemma \ref{prop:Nt and Measure}.\textcolor{black}{We notice that,$r_t - q(1)\int_{0}^{t}(1-\xi_s(-1))\,ds $ and $l_t - q(-1)\int_{0}^{t}(1-\xi_s(1))\,ds $ are $\mathbb{P}^{\nu_0,q}$- martingales, and that the left hand side of \eqref{Easy Estimate2} can be rewritten as
	\[q(1)\left<\bar{\nu},\xi_{-1}-\xi_{1}\right> - (q(-1)-q(1))\left<\bar{\nu},1-\xi_{-1}\right>.\]}
\end{proof}

\textcolor{black}{We use $\liminf$ in \eqref{Easy Estimate2} to emphasis that the initial measure $\nu_0$ is arbitrary, and $D_t$ may not have a law of large numbers at this stage.} From the estimate \eqref{est: lowerbdd for displacement} in Lemma \ref{prop:displacement and spped of tagged}, we can get a positive mean for the displacement when the tagged particle has almost symmetric jump rates, \textcolor{black}{i.e.} when $q(-1)-q( 1)$ is small, and $C_0$ is positive.
For the next three sections, we will show how to get a positive $C_0$ with \eqref{est:lower bound for current} in Lemma \ref{prop:Nt and Measure} and \ref{prop:displacement and spped of tagged} \textcolor{black}{for some $\nu_0$}.
	%
	

\section{An Error Estimate and Couplings of Particles on $\mathbb{Z}$} \label{sec:coupling}


The main result of this section is Theorem \ref{thm:replacing estimate}, which gives an estimate of the error $\mathbb{E}^{\nu_0,q}\left[N_t\right]-\mathbb{E}^{\nu_0,0}\left[N_t\right]$. \textcolor{black}{This estimate allows us to consider the problem with a fixed tagged particle instead of a moving tagged particle.} The proof relies on couplings of two auxiliary processes\textcolor{black}{, which is the main tool in this section. The couplings are similar to those in \cite{Ba, FLS}. We will order particles in increasing order, and compare the positions of particles in two processes in pairs. Typically, a particle in the "faster" process have larger coordinates than its paired particle in the "slower" process. By coupling jumps of particles, we will preserve the relative orders of paired particles in the processes.}  \textcolor{black}{Next, we introduce some notions, and show the proof of Theorem \ref{thm:replacing estimate} at the end of this section.} 

\subsection{Auxiliary Processes}
We can view the environment process $\xi_t$ of the asymmetric exclusion process with a tagged particle in another way. We can label all red particles according to the initial configuration in an ascending order, and track their relative positions with respect to the tagged particle.

Starting from an initial configuration $\xi$ with infinitely many particles on both sides of zero, we label particles with their initial positions as $\vec{X}_0 = (X_i)_{i\in \mathbb{Z}} \in \mathbb{Z}^{\mathbb{Z}}=\tilde{\mathbb{X}}$. Particularly, $\vec{X}_0$ satisfies
\begin{equation} \label{Cond:ascending inital configuration}
\dots < X_{-2}<X_{-1}<X_0 < X_1 <X_2 <\dots 
\end{equation}
and
\[
\xi(x) = 1   \Leftrightarrow  X_i= x, \text{for some }i.
\]

\textcolor{black}{To extend to the case when} there are finitely many particles to the right or the left of zero, it is also convenient for us to add particles at $+\infty$ and $-\infty$, and therefore, we would enlarge the state space to $\hat{\mathbb{X}} = (\mathbb{Z}\bigcup\{-\infty,\infty\})^{\mathbb{Z}}$. For example, for the step measure $\mu_{1,0}$, we can label particles as:
\[ \dots < X_{-2}=-3<X_{-1}=-2 <X_0=-1 < X_1 =\infty \leq X_2 = \infty \leq \dots \] \textcolor{black}{Also, there is no particular rule for the choice of $X_0$.}

\textcolor{black}{For each initial configuration $\vec{X}_0$ satisfying \eqref{Cond:ascending inital configuration}, there is a Markov process $\vec{X}_t$ with generator $\tilde{\mathit{L}}$ corresponding to the process $\xi_t$ with initial configuration $\xi$. Particularly, for any $t>0$, $\vec{X}_t$ also satisfies \eqref{Cond:ascending inital configuration}. Note that there are multiple $\vec{X}_0$ corresponding to $\xi$, so to the process $\xi_t$ there correspond multiple processes $\vec{X}_t$.} There are also two types of jumps for the auxiliary process, \textcolor{black}{corresponding} to jumps \eqref{usual exchange} and \eqref{usual translation}. The first occurs when the i-th particle jumps to an empty target site $X_i +z$; the second occurs when the tagged particle jumps to an empty target site $z$. Let $T_{i,z}\vec{X}$ and $\Theta_z\vec{X}$ represent the configurations after these two jumps respectively. We can see two types of jumps in Figure \ref{Figure 1} and Figure \ref{Figure 2}. 

\begin{figure}[h!]
	\begin{center}
		\begin{tikzpicture}
		\draw[black, thick] (-4.3,0) -- (4.3,0) node[black,right] {$\vec{X}$};  
		\draw [black,fill] (-0.1,0.1) rectangle (0.1,-0.1) node [black,below=4] {Tagged}; 
		\foreach \x in {-4,...,4}
		\draw (-\x,0) circle (0.1);
		\draw [black,fill] (-4,0) circle [radius=0.1] node [black,below=4] {$X_{-3}$}; 
		\draw [black,fill] (-3,0) circle [radius=0.1] node [black,below=4] {$X_{-2}$}; 
		\draw [black,fill] (-1,0) circle [radius=0.1] node [black,below=4] {$X_{-1}$}; 
		\draw [black,fill] (2,0) circle [radius=0.1] node [black,below=4] {$X_{0}$}; 
		
		\draw [dashed,->] (-3.9,0.3) 
		to [out=45,in=181] (-3,0.8)
		to [out=-1,in=135] (-2.1,0.3);
		
		\draw[black, thick] (-4.3,-1) -- (4.3,-1) node[black,right] {$T_{-3,2}\vec{X}$};  
		\draw [black,fill] (-0.1,0.1-1) rectangle (0.1,-0.1-1) node [black,below=4] {Tagged}; 
		
		\foreach \x in {-4,...,4}
		\draw (-\x,0-1) circle (0.1);
		\draw [black,fill] (-3,0-1) circle [radius=0.1] node [black,below=4] {$X_{-3}$}; 
		\draw [black,fill] (-2,0-1) circle [radius=0.1] node [black,below=4] {$X_{-2}$}; 
		\draw [black,fill] (-1,0-1) circle [radius=0.1] node [black,below=4] {$X_{-1}$}; 
		\draw [black,fill] (2,0-1) circle [radius=0.1] node [black,below=4] {$X_{0}$}; 
		\end{tikzpicture}
	\end{center}
	\caption{Red Particle $X_{-3}$ Jumps 2 Units}
	\label{Figure 1}
\end{figure}
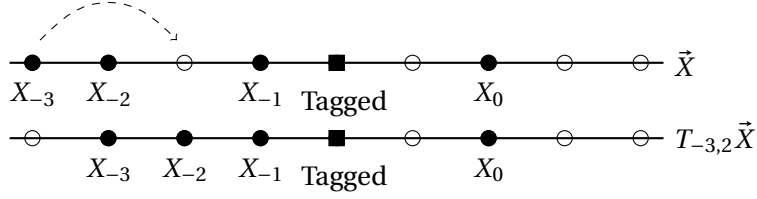

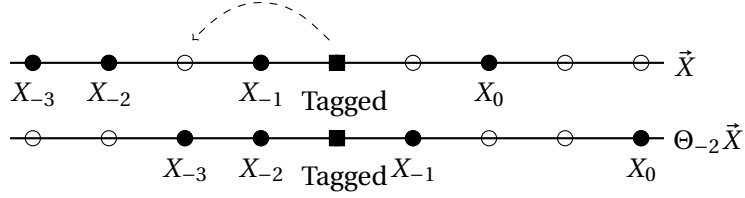
\begin{figure}[h!] 
	\begin{center}
		\begin{tikzpicture} 
		\draw[black, thick] (-4.3,0) -- (4.3,0)node[black,right] {$\vec{X}$};  
		\draw [black,fill] (-0.1,0.1) rectangle (0.1,-0.1) node [black,below=4] {Tagged}; 
		
		\foreach \x in {-4,...,4}
		\draw (-\x,0) circle (0.1);
		\draw [black,fill] (-4,0) circle [radius=0.1] node [black,below=4] {$X_{-3}$}; 
		\draw [black,fill] (-3,0) circle [radius=0.1] node [black,below=4] {$X_{-2}$}; 
		\draw [black,fill] (-1,0) circle [radius=0.1] node [black,below=4] {$X_{-1}$}; 
		\draw [black,fill] (2,0) circle [radius=0.1] node [black,below=4] {$X_{0}$}; 
		
		\draw [dashed,->] (-0.1,0.3) 
		to [out=135,in=-1] (-1,0.8)
		to [out=181,in=45] (-1.9,0.3);

		\draw[black, thick] (-4.3,-1) -- (4.3,-1) node[black,right] {$\Theta_{-2}\vec{X}$};  
		\draw [black,fill] (-0.1,0.1-1) rectangle (0.1,-0.1-1) node [black,below=4] {Tagged}; 
		
		\foreach \x in {-4,...,4}
		\draw (-\x,0-1) circle (0.1);
		\draw [black,fill] (-2,0-1) circle [radius=0.1] node [black,below=4] {$X_{-3}$}; 
		\draw [black,fill] (-1,0-1) circle [radius=0.1] node [black,below=4] {$X_{-2}$}; 
		\draw [black,fill] (1,0-1) circle [radius=0.1] node [black,below=4] {$X_{-1}$}; 
		\draw [black,fill] (4,0-1) circle [radius=0.1] node [black,below=4] {$X_{0}$}; 
		\end{tikzpicture}
	\end{center}
	\caption{Tagged Particle Jumps -2 Units}
	\label{Figure 2}
\end{figure}

\textcolor{black}{
For any $z\neq 0$, we have $\Theta_z\vec{X}$ as,
\begin{equation}
(\Theta_z\vec{X})_j =X_j-z.  \label{auxiliary translation}
\end{equation}
}
\textcolor{black}{
For $z>0$, we denote the index of the right-most particle to the left of site $X_i+ z$ by  $I_{i,z}(\vec{X})$,
\begin{equation}
I_{i,z}(\vec{X}) = \max\{k:X_k \leq X_i+z\}. \label{eq:index of right jump} 
\end{equation}}

When positive jump of size $z$ is possible for the i-th particle, the new configuration after jump is
\begin{align}
(T_{i,z}\vec{X})_j =&
\begin{cases}
X_j & \text{if }j<i \text{ or } j >I_{i,z}(\vec{X})\\
X_{j+1} & \text{if } i\leq j <I_{i,z}(\vec{X}) \\
X_i+z   & \text{if } j = I_{i,z}(\vec{X})
\end{cases}\label{auxiliary exchange}.\
\end{align}
The conditions for these two types of jumps are $B_z=\{z \notin \vec{X}\}$, and $A_{i,z}= \{X_i+z \notin \vec{X}\cup \{0\}\}$, respectively. Here we also think \textcolor{black}{of} $\vec{X}$ as a subset of $\mathbb{Z}$ (instead of $\mathbb{Z}\cup{\{-\infty,\infty\}}$).

\begin{remark} \label{rm:cov and reverse}
	For negative jumps $z<0$, we can think of the \textcolor{black}{dynamics} by reversing the lattice $\mathbb{Z}$. That is, with a change of variable, $\vec{Y} = \{Y_i\}_{i\in\mathbb{Z}} = R(\vec{X})$, we have 
	\textcolor{black}{
	\begin{equation}
	\left(R(\vec{X})\right)_{i}=Y_i=-X_{-i}
	\end{equation}
	\begin{equation} \label{eq: reverse T_i,z}
	(T_{i,z}\vec{X}) = R(T_{-i,-z}(R(\vec{X}))
	\end{equation}
	\begin{equation} \label{eq: reverse I_i,z}
	I_{i,z}(\vec{X}) = - I_{-i,-z}(R(\vec{X}))=\min\{k: X_k\geq X_i+z\} 
	\end{equation}}
	For $z=0$, we take $T_{i,0}$ as the identity map and $I_{i,0}=i$.
\end{remark}

Therefore, we can write down the generator $\tilde{\mathit{L}}$ for the auxiliary process $\vec{X}_t$ by its action on local functions \textcolor{black}{$F: \hat{\mathbb{X}}\longrightarrow \mathbb{R}$ (i.e. $F(\vec{X})$ depends on a finite set $\{X_i\}$) as}:

\begin{align}
\tilde{\mathit{L}}F(\vec{X}) =& (\tilde{\mathit{L}}^{ex}+\tilde{\mathit{L}}^{sh}) F(\vec{X}) \notag \\
=&\sum_{i,z}p(X_i,X_i+z)\mathbb{1}_{A_{i,z}}(\vec{X})\left[F(T_{i,z}\vec{X})-F(\vec{X})\right] \notag \\
+&\sum_{y}q(y)\mathbb{1}_{B_{y}}(\vec{X})\left[F(\Theta_y\vec{X})-F(\vec{X})\right].  \label{Eq:auxiliary generator}
\end{align}
The transition rates \textcolor{black}{are} $p(x,y) = p(y-x)$ if $x,y \neq 0,\pm\infty$, and $p(x,y) = 0$ \textcolor{black}{otherwise}.

\subsection{Shifts of Labels}
\textcolor{black}{In the environment process, a jump of the tagged particle influences coordinates of all red particles while jumps of red particles influence only finitely many coordinates. This may create a problem for the coupling. For example, when only the tagged particle in the slower system jumps towards left and the tagged particle in the faster system stays, the coordinates of all red particles in the slower system increase but those in the faster system stay the same. In order to offset this type of global effects from of the tagged particle and preserve the order of two processes, we use shifts of labels.} 

\textcolor{black}{For couplings, we also consider two other versions of auxiliary processes with shifts of labels.}  Let $S_z\vec{X}$ represents the configuration after shifting labels by $z$,
\begin{equation} \label{shift labels}
(S_z\vec{X})_j = X_{j+z}
\end{equation}

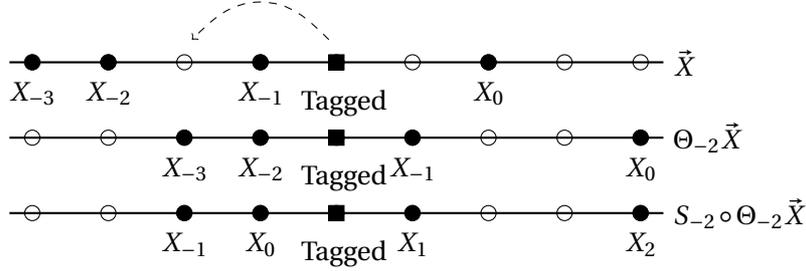
\begin{figure}[h!]
	\begin{center}
		\begin{tikzpicture} 
		\draw[black, thick] (-4.3,0) -- (4.3,0)node[black,right] {$\vec{X}$};  
		\draw [black,fill] (-0.1,0.1) rectangle (0.1,-0.1) node [black,below=4] {Tagged}; 
		
		\foreach \x in {-4,...,4}
		\draw (-\x,0) circle (0.1);
		\draw [black,fill] (-4,0) circle [radius=0.1] node [black,below=4] {$X_{-3}$}; 
		\draw [black,fill] (-3,0) circle [radius=0.1] node [black,below=4] {$X_{-2}$}; 
		\draw [black,fill] (-1,0) circle [radius=0.1] node [black,below=4] {$X_{-1}$}; 
		\draw [black,fill] (2,0) circle [radius=0.1] node [black,below=4] {$X_{0}$}; 
		
		\draw [dashed,->] (-0.1,0.3) 
		to [out=135,in=-1] (-1,0.8)
		to [out=181,in=45] (-1.9,0.3);

		\draw[black, thick] (-4.3,-1) -- (4.3,-1) node[black,right] {$\Theta_{-2}\vec{X}$};  
		\draw [black,fill] (-0.1,0.1-1) rectangle (0.1,-0.1-1) node [black,below=4] {Tagged}; 
		
		\foreach \x in {-4,...,4}
		\draw (-\x,0-1) circle (0.1);
		\draw [black,fill] (-2,0-1) circle [radius=0.1] node [black,below=4] {$X_{-3}$}; 
		\draw [black,fill] (-1,0-1) circle [radius=0.1] node [black,below=4] {$X_{-2}$}; 
		\draw [black,fill] (1,0-1) circle [radius=0.1] node [black,below=4] {$X_{-1}$}; 
		\draw [black,fill] (4,0-1) circle [radius=0.1] node [black,below=4] {$X_{0}$}; 

		\draw[black, thick] (-4.3,-2) -- (4.3,-2) node[black,right] {$S_{-2}\circ\Theta_{-2}\vec{X}$};  
		\draw [black,fill] (-0.1,0.1-2) rectangle (0.1,-0.1-2) node [black,below=4] {Tagged}; 
		
		\foreach \x in {-4,...,4}
		\draw (-\x,0-2) circle (0.1);
		\draw [black,fill] (-2,0-2) circle [radius=0.1] node [black,below=4] {$X_{-1}$}; 
		\draw [black,fill] (-1,0-2) circle [radius=0.1] node [black,below=4] {$X_{0}$}; 
		\draw [black,fill] (1,0-2) circle [radius=0.1] node [black,below=4] {$X_{1}$}; 
		\draw [black,fill] (4,0-2) circle [radius=0.1] node [black,below=4] {$X_{2}$}; 
		
		\end{tikzpicture}
	\end{center}
	\caption{Tagged Particle Jumps -2 Units with Labels Shifted}
	\label{Figure 3}
\end{figure}

In addition to shifting configurations when a tagged particle jumps, we also shift labels after shifting the configurations. See Figure \ref{Figure 3}. We obtain the first version by adding a shift of labels by z after the tagged particle has a right jump with z units, that is,
\begin{align}
\tilde{\mathit{L}}_RF(\vec{X}) =& (\tilde{\mathit{L}}^{ex}+\tilde{\mathit{L}}^{sh,q_-}+\tilde{\mathit{L}}_R^{sh,q_+}) F(\vec{X}) \notag \\
=&\sum_{i,z}p(X_i,X_i+z)\mathbb{1}_{A_{i,z}}(\vec{X})\left[F(T_{i,z}\vec{X})-F(\vec{X})\right] \notag \\
+&\sum_{y<0}q(y)\mathbb{1}_{B_{y}}(\vec{X})\left[F(\Theta_y\vec{X})-F(\vec{X})\right] \notag \\
+&\sum_{y>0}q(y)\mathbb{1}_{B_{y}}(\vec{X})\left[F(S_y\circ\Theta_y\vec{X})-F(\vec{X})\right] \label{Eq:auxiliary generator-Right}
\end{align}

Similarly, we can have the second version by shifting labels after the tagged particle takes a left jump.
\begin{align}
\tilde{\mathit{L}}_LF(\vec{X}) =& (\tilde{\mathit{L}}^{ex}+\tilde{\mathit{L}}_L^{sh,q_-}+\tilde{\mathit{L}}^{sh,q_+}) F(\vec{X}) \notag \\
=&\sum_{i,z}p(X_i,X_i+z)\mathbb{1}_{A_{i,z}}(\vec{X})\left[F(T_{i,z}\vec{X})-F(\vec{X})\right] \notag \\
+&\sum_{y<0}q(y)\mathbb{1}_{B_{y}}(\vec{X})\left[F(S_y\circ\Theta_y\vec{X})-F(\vec{X})\right] \notag \\
+&\sum_{y>0}q(y)\mathbb{1}_{B_{y}}(\vec{X})\left[F(\Theta_y\vec{X})-F(\vec{X})\right] \label{Eq:auxiliary generator-Left}
\end{align}

We will use $\vec{X}_t = (\vec{X}_0,G,p,q)$ to denote the auxiliary process with $\vec{X}_0$ as the initial configuration, and generator $G$. Particularly, G is one of the forms \eqref{Eq:auxiliary generator},\eqref{Eq:auxiliary generator-Right}, and \eqref{Eq:auxiliary generator-Left} with $p,q$ as parameters. And we use $\mathbb{P}^{(\vec{X}_0,G,p,q)}$ or $\mathbb{P}^{\vec{X}_t}$ to denote the corresponding probability measure on the space of c\'{a}dl\'{a}g paths on $\hat{\mathbb{X}}$. $\vec{X_0}$ can also be random.

\subsection{Couplings of Auxiliary Processes and Error Estimates}
There is a natural partial order on the set $\hat{\mathbb{X}}$:
\begin{equation}\label{eq:partial oder}
\vec{X} \geq \vec{Y} \Leftrightarrow  X_i \geq Y_i, \text{ for all } i.  \end{equation}
With this partial order, 
we can use stochastic ordering to define couplings of two auxiliary processes $\vec{X}_t= (\vec{X}_0,G,p,q)$ and $\vec{Y}_t=(\vec{Y}_0,G',p',q')$.
\begin{definition}\label{def:coupling}
	\textcolor{black}{We denote $\vec{X}_t \succeq \vec{Y}_t$, if two auxiliary proccesses $\vec{X}_t$ and $\vec{Y}_t$ can be coupled:} that is, there exists a joint process $\vec{Z}_t = (\vec{W}_t,\vec{V}_t)$, with a joint generator  $\Omega$ on space of local functions $F:\tilde{\mathbb{X}}\times\tilde{\mathbb{X}}\mapsto \mathbb{R}$, such that
	\begin{enumerate}
		\item $\vec{W}_t \geq \vec{V}_t, \mathbb{P}^{\vec{Z}_t}\--a.s.$
		\item $\vec{Z}_t$ has marginals as $\vec{X}_t$ and $\vec{Y}_t$. That is, for any local functions $F_1(\vec{X},\vec{Y})= H_1(\vec{X})$, and $F_2(\vec{X},\vec{Y})= H_2(\vec{Y})$, we have,
		\begin{align*}
		\Omega F_1(\vec{X},\vec{Y})&= G H_1(\vec{X})\\
		\Omega F_2(\vec{X},\vec{Y})&= G' H_2(\vec{Y})\\
		\vec{W}_0 \overset{d}{=} \vec{X}_0&,\  \vec{V}_0 \overset{d}{=} \vec{Y}_0
		\end{align*}
	\end{enumerate}  
\end{definition}

Our main step towards Theorem \ref{thm:replacing estimate} is the existence of couplings of auxiliary processes. The construction of the couplings is done in Appendix A.
\begin{theorem}\label{thm:coupling}
	Let $p(\cdot)$ satisfy assumption \ref{Asmp:attractiveness} and two initial configurations satisfy ${\vec{X}_0 \geq \vec{Y}_0}$. For any $q(\cdot)$, we can couple below two pairs of auxiliary processes:
	\begin{align}
	(\vec{X}_0,\tilde{L}_R,p,q) &\succeq ( \vec{Y}_0,\tilde{L},p,0) \\
	(\vec{X}_0,\tilde{L},p,0) &\succeq (\vec{Y}_0,\tilde{L}_L,p,q)
	\end{align}
\end{theorem}

\begin{proof}
	See Theorem \ref{Thm: main step} in Appendix A.
\end{proof}

 Above two couplings provide a lower bound and an upper bound of the error $\mathbb{E}^{\nu_0,q}\left[N_t\right]-\mathbb{E}^{\nu_0,0}\left[N_t\right]$ respectively, and we can estimate the error by the number of jumps of the tagged particle.
\begin{theorem} \label{thm:replacing estimate}
	Let $p(\cdot)$ satisfy assumption \ref{Asmp:attractiveness}, and the tagged particle takes nearest-neighbor jumps, with rates $q(-1),q(1)$. For any (deterministic) initial configuration  $\xi$, and any $t\geq 0$,
	\begin{equation} \label{eq:error estimate}
	\abs{\mathbb{E}^{\xi,q}\left[N_t\right]-\mathbb{E}^{\xi,0}\left[N_t\right]} \leq t\cdot(q(1) + q(-1))
	\end{equation}
\end{theorem}

\begin{proof} 
	For any non-zero configuration $\xi$ in $\mathbb{X}$, we can label the particles as $\vec{X_0} = \{X_i\}_{i\in \mathbb{Z}}$, 
	\[\dots \leq X_{-2}\leq X_{-1}\leq X_0 < 0 < X_1 \leq X_2 \leq \dots \]
	and equality occurs if both sides are $\infty$ or $-\infty$. By Theorem \ref{thm:coupling}, from the same initial configuration, we have two couplings, 
	\begin{align} \label{eq: important coupling}
	\vec{X}_t=(\vec{X}_0,\tilde{L}_R,p,q) &\succeq ( \vec{X}_0,\tilde{L},p,0) = \vec{Y}_t \notag\\
	\vec{Y}_t =(\vec{X}_0,\tilde{L},p,0) &\succeq (\vec{X}_0,\tilde{L}_L,p,q) = \vec{Z}_t.
	\end{align}
	
	Consider a function $F: \hat{X} \to \mathbb{Z}$, $F(\vec{X}) = \max\{i:X_i \leq -1\}$. It is decreasing in $\vec{X}$, that is, if $\vec{X} \geq \vec{Y}$
	\begin{equation} \label{monotone function}
	F(\vec{X}) \leq F(\vec{Y})
	\end{equation}
	
	Therefore, we get, under two joint distributions \textcolor{black}{(one for the coupling $\vec{X}_t\succeq \vec{Y}_t$, and the other for the coupling $\vec{Y}_t\succeq \vec{Z}_t$) and $\vec{X}_0=\vec{Y}_0=\vec{Z}_0$,}
	\begin{align}
	F(\vec{X}_0)- F(\vec{X}_t) &\geq F(\vec{Y_0})-F(\vec{Y}_t),  a.s., \label{ineq:index change1} \\F(\vec{Y}_0)- F(\vec{Y}_t) &\geq F(\vec{Z_0})-F(\vec{Z}_t),  a.s.. \label{ineq:index change2}
	\end{align}
	
	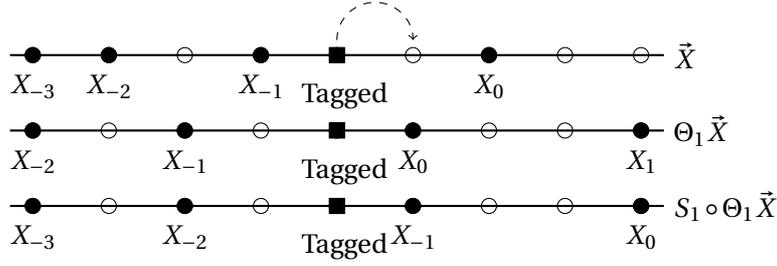
\begin{figure}[h!]
		\begin{center}
			\begin{tikzpicture} 
			\draw[black, thick] (-4.3,0) -- (4.3,0)node[black,right] {$\vec{X}$};  
			\draw [black,fill] (-0.1,0.1) rectangle (0.1,-0.1) node [black,below=4] {Tagged}; 
			
			\foreach \x in {-4,...,4}
			\draw (-\x,0) circle (0.1);
			\draw [black,fill] (-4,0) circle [radius=0.1] node [black,below=4] {$X_{-3}$}; 
			\draw [black,fill] (-3,0) circle [radius=0.1] node [black,below=4] {$X_{-2}$}; 
			\draw [black,fill] (-1,0) circle [radius=0.1] node [black,below=4] {$X_{-1}$}; 
			\draw [black,fill] (2,0) circle [radius=0.1] node [black,below=4] {$X_{0}$}; 
			
			\draw [dashed,->] (0,0.2) arc (180:0:0.5);

			\draw[black, thick] (-4.3,-1) -- (4.3,-1) node[black,right] {$\Theta_{1}\vec{X}$};  
			\draw [black,fill] (-0.1,0.1-1) rectangle (0.1,-0.1-1) node [black,below=4] {Tagged}; 
			
			\foreach \x in {-4,...,4}
			\draw (-\x,0-1) circle (0.1);
			\draw [black,fill] (-4,0-1) circle [radius=0.1] node [black,below=4] {$X_{-2}$}; 
			\draw [black,fill] (-2,0-1) circle [radius=0.1] node [black,below=4] {$X_{-1}$}; 
			\draw [black,fill] (1,0-1) circle [radius=0.1] node [black,below=4] {$X_{0}$}; 
			\draw [black,fill] (4,0-1) circle [radius=0.1] node [black,below=4] {$X_{1}$}; 

			\draw[black, thick] (-4.3,-2) -- (4.3,-2) node[black,right] {$S_{1}\circ\Theta_{1}\vec{X}$};  
			\draw [black,fill] (-0.1,0.1-2) rectangle (0.1,-0.1-2) node [black,below=4] {Tagged}; 
			
			\foreach \x in {-4,...,4}
			\draw (-\x,0-2) circle (0.1);
			\draw [black,fill] (-4,0-2) circle [radius=0.1] node [black,below=4] {$X_{-3}$}; 
			\draw [black,fill] (-2,0-2) circle [radius=0.1] node [black,below=4] {$X_{-2}$}; 
			\draw [black,fill] (1,0-2) circle [radius=0.1] node [black,below=4] {$X_{-1}$}; 
			\draw [black,fill] (4,0-2) circle [radius=0.1] node [black,below=4] {$X_{0}$}; 

			\end{tikzpicture}
		\end{center}
		\caption{Tagged Particle Jumps 1 Unit with Labels Shifted}
		\label{Figure 4}
	\end{figure}
	
	On the other hand, when $q(\cdot)$ is nearest-neighbor, jumps of tagged particle do not move particles between positive and negative axes, \textcolor{black}{but they may shift labels.} See Figure \ref{Figure 4}. \textcolor{black}{By a decomposition similar to those in \eqref{eq: Nt} and \eqref{eq: Dt}, we can see that the change in the label of the right-most particle on negative axis by time t comes from three sources: jumps of red particles through bond $(-1,1)$ ($N_{\vec{X}}(t)$), right jumps of the tagged particle ($r_{\vec{X}}(t)$) and left jumps of the tagged particle ($l_{\vec{X}}(t)$). We obtain}
	\begin{equation} \label{eq:change in index_x}
	F(\vec{X}_0)- F(\vec{X}_t) = N_{\vec{X}}(t)+\textcolor{black}{r}_{\vec{X}}(t), 
	\end{equation}
	\textcolor{black}{where $N_{\vec{X}}(t),{r}_{\vec{X}}(t)$  are the same as $N_t, r_t$ for the auxiliary process $\vec{X}_t$. We also obtain two similar identities for processes $\vec{Y}_t, \vec{Z}_t$},
	\begin{align}
	F(\vec{Y}_0)- F(\vec{Y}_t) &= N_{\vec{Y}}(t), \label{eq:change in index_y}\\
	F(\vec{Z}_0)- F(\vec{Z}_t) &= N_{\vec{Z}}(t)-\textcolor{black}{l}_{\vec{Z}}(t),\label{eq:change in index_z}
	\end{align} 
	\textcolor{black}{where $l_{\vec{Z}}(t)$ is the same as $l_t$ for process $\vec{Z}(t)$.}
	
\textcolor{black}{Taking expectations on \eqref{ineq:index change1} rewritten in terms of \eqref{eq:change in index_x} and  \eqref{eq:change in index_y}}, we get,
	\begin{equation}
	\mathbb{E}^{\nu_0,q}\left[N_t\right]-\mathbb{E}^{\nu_0,0}\left[N_t\right] \geq -\mathbb{E}^{\nu_0,q}\left[r_t\right]; \label{est:lower bdd for current difference}
	\end{equation}
	\textcolor{black}{taking expectations on \eqref{ineq:index change2} rewritten in terms of \eqref{eq:change in index_y} and  \eqref{eq:change in index_z}, we get}
	\begin{equation}
	\mathbb{E}^{\nu_0,q}\left[N_t\right]-\mathbb{E}^{\nu_0,0}\left[N_t\right] \leq \mathbb{E}^{\nu_0,q}\left[l_t\right]. \label{est:upper bdd for current difference}
	\end{equation}
	\eqref{est:lower bdd for current difference} and \eqref{est:upper bdd for current difference}  are sufficient for \eqref{eq:error estimate}.
\end{proof}

\begin{remark}\label{rm: similar proofs}
	We can obtain further results with similar proofs of Theorem \ref{thm:replacing estimate}. We mention them without detailed proofs.
	\begin{enumerate}
		\item There is a similar estimate when $q(\cdot)$ is non-nearest-neighbor. We only need to add some terms which are multiples of  $ \mathbb{E}^{\nu_0,q}\left[r_t\right]$ and $ \mathbb{E}^{\nu_0,q}\left[l_t\right]$ to the right hand sides of \eqref{est:upper bdd for current difference} and \eqref{est:lower bdd for current difference}. We can find a $C_{R'}$ depending on the range $R'$ of $q(\cdot)$ such that,
		\[\abs{\mathbb{E}^{\nu_0,q}\left[N_t\right]-\mathbb{E}^{\nu_0,0}\left[N_t\right]} \leq C_{R'}\textcolor{black}{\sum_z q(z)\cdot t.}\]
		\item From the coupling, we can use Kingman's Subadditive Ergodic Theorem to show the convergence of $\frac{N_t}{t}$ when the initial measure is the step measure $\mu_{1,0}$, and the tagged particle does not move, $q =0$:
		\[\lim_{t\to \infty} \frac{N_t}{t} = \lim_{t\to \infty} \frac{1}{t}\mathbb{E}^{\mu_{1,0},0}\left[N_t\right],\quad \mathbb{P}^{\mu_{1,0},0}-a.s. \]\textcolor{black}{See Remark \ref{rm: hydro and baha} and \cite{Ba}}.
		
		\item From the coupling and \textcolor{black}{Remark \ref{rm: similar proofs}.2, we can get a lower bound similar to \eqref{est: lowerbdd for displacement}} for the displacement of a tagged particle, when the initial meausre is the step measure $\mu_{1,0}$, and $q(\cdot)$ is nearest-neighbor,
		\[ \liminf_{t\to \infty} \frac{D_t}{t} \geq \frac{q(1)}{p(2)}(C_1 -q(-1)) -(q(-1)-q(1)),\quad \mathbb{P}^{\mu_{1,0},q}-a.s.   \]
		where $C_1$ is from  \textcolor{black}{Remark\ref{rm: similar proofs}.2},
		\[C_1=\lim_{t\to \infty} \frac{1}{t}\mathbb{E}^{\mu_{1,0},0}\left[N_t\right].\]
		
	\end{enumerate}
\end{remark}

\section{Current in AEP with a Blockage} \label{sec: Current Zero}
In this section, we will show the current in AEP with a blockage at the origin has a positive lower bound \textcolor{black}{(Theorem \ref{thm:positive current when it is not fully bocked}).} \textcolor{black}{The existence of a positive lower bound helps us to show that the tagged particle has a positive speed under $\mathbb{P}^{\nu_e,q},$ for some small $q(\cdot)$ and some ergodic measure $\nu_e$.}
\subsection{Currents and Densities in Equilibrium}
\textcolor{black}{In Sections \ref{sec: Current Zero}, \ref{sec: ASEP on the half line with destruction}}, we make the following assumptions on $p(\cdot,\cdot)$. Let $p(\cdot,\cdot)$ be jump rates for a continuous-time random walk on $\mathbb{Z}$ with the following conditions:
\begin{enumerate} 
	\item $p(\cdot,\cdot) $ is translational invariant: $p(x,y) = p(y-x)$.
	\item $p(x,x+k)=p(k)\geq p(-k) = p(x+k,x)$ for all \textcolor{black}{$k>0$}, and a strict inequality holds for some k.
	\item $p(\cdot,\cdot) $ has a finite jump range $R>1$:  $p(k)=0, \abs{k} > R$. Assume further $p(R) >0$.
\end{enumerate}
\textcolor{black}{
We don't need \ref{Asmp:attractiveness2}, which is the main condition for the existence of couplings in Section \ref{sec:coupling}; instead, the second condition above enables us to construct an increasing sequence $G_i$, which will be important in the the proof of Lemma \ref{lm: current 0 and asymptotic density}.}

We will consider a process, the AEP on lattice $\mathbb{Z}$ with 
a blockage at the origin, \textcolor{black}{i.e. the AEP with a tagged particle when $q=0$,} and quantities $C_{x,y}$ that are currents \textcolor{black}{through} bond \textcolor{black}{$(x,y)$}. 

The AEP on lattice $\mathbb{Z}$ with 
a blockage at the origin has a generator $L$ \textcolor{black}{defined by its action on a local function $f$},
\begin{equation}\label{eq: generator for Blocked ASEP}
\mathit{L}f(\eta) 
=\sum_{x,y\neq 0}p(x,y)\eta_x\left( 1-\eta_y\right)\left(f(\eta^{x,y})-f(\eta)\right), 
\end{equation}
\textcolor{black}{which is the same as \eqref{Eq:generator} when $q=0$.}
Assume the initial configuration is the step measure $\mu_{1,0}$ for the rest of this section. The current $C_{i,j}$ \textcolor{black}{through} bond $(i,j)$ is defined as:

\textcolor{black}{
\begin{equation} \label{eq:current}
C_{x,y}(\eta) = \sum_{\substack{i\leq x, y\leq j,\\ i,j\neq 0}}p(i,j)\eta_i(1-\eta_j) - p(j,i) \eta_j(1-\eta_i).
\end{equation}
}

Theorem \ref{thm:positive current when it is not fully bocked} is the main result for the next two sections. Before its \textcolor{black}{statement and proof}, we shall see three lemmas on invariant measures with respect to $\mathit{L}$, and currents $C_{i,j}$. \textcolor{black}{The first two lemmas are direct consequences of translation invariance and finite range of $p(\cdot,\cdot)$ and they are standard. In the third lemma, we will need the second condition on $p(\cdot,\cdot)$.} The first lemma says the mean of current $C_{x,x+1}$ is constant in $x$ with respect to an invariant measure.
\begin{lemma}\label{lm: constant current for invariant measure}
	For an invariant measure $\bar{\nu}$ with respect to the generator $L$ defined in \eqref{eq: generator for Blocked ASEP}, we have, for any $x \neq -1, 0$,
	\begin{equation} \label{eq: constant current for invariant measure}
	\langle\bar{\nu},C_{x,x+1}\rangle =\langle\bar{\nu},C_{-1,1}\rangle
	\end{equation}
\end{lemma}
\begin{proof} The change of density at site $x$ is due to the difference between currents \textcolor{black}{through} bonds $(x-1,x)$ and $(x,x+1)$. 
	Computing $\mathit{L} \eta_x$ for $x\neq -1,0,1$, we get
	\begin{align*}
	\mathit{L} \eta_x &= C_{x-1,x}- C_{x,x+1}, \\
	\mathit{L} \eta_{-1} &= C_{-2,-1}-C_{-1,1}, \\
	\mathit{L} \eta_{1} &= C_{-1,1}-C_{1,2}.
	\end{align*}
	We show the first one, and the rest two are similar:
	\begin{align*}
	\mathit{L} \eta_x =& \sum_{i,j\neq 0}p(i,j)\eta_i\left( 1-\eta_j
	\right)\left(\eta^{i,j}_x-\eta_x\right)\\= & \sum_{i\neq 0,x}p(i,x)\eta_i\left( 1-\eta_x\right) - \sum_{j\neq 0,x}p(x,j)\eta_x\left( 1-\eta_j\right)  \\
	=& \sum_{i\neq 0,x}\left\lbrace p(i,x)\eta_i\left( 1-\eta_x\right) - p(x,i)\eta_x\left( 1-\eta_i\right)\right\rbrace	
	\end{align*}
	On the other hand,
	\begin{align*}
	C_{x-1,x} -C_{x,x+1} =& \left(\sum_{\substack{i\leq x-1, x\leq j,\\ i,j\neq 0}}-\sum_{\substack{i\leq x, x+1\leq j,\\ i,j\neq 0}}\right)  p(i,j)\eta_i(1-\eta_j) - p(j,i) \eta_j(1-\eta_i) \\
	=& 	\left(\sum_{\substack{i\leq x-1, x= j,\\ i,j\neq 0}} - \sum_{\substack{i= x, x+1\leq j,\\ i,j\neq 0}} \right) p(i,j)\eta_i(1-\eta_j) - p(j,i) \eta_j(1-\eta_i) \\
	=& \textcolor{black}{	\left(\sum_{\substack{i\leq x-1,\\ i\neq 0}} + \sum_{\substack{ x+1\leq i,\\ i\neq 0}} \right)p(i,x)\eta_i(1-\eta_x) - p(x,i) \eta_x(1-\eta_i)	= L\eta_x.}
	\end{align*}
	The third line is because interchanging $i$ and $j$ results in a change of sign. 
	
	Taking expectation with respect to $\bar{\nu}$, we get \eqref{eq: constant current for invariant measure}.
\end{proof}

Consider translation operators $\tau_i$ on the state space $\mathbb{X'}= \{0,1\}^{\mathbb{Z}}$ , for $i,j\in\mathbb{Z}$,
\[(\tau_i\eta)(j)=\eta(j+i).\]
We define translations on local functions $f$ and on measures $\nu$ by
\begin{equation} 
\tau_i f(\eta) = f(\tau_i \eta),\label{eq:translation of function}
\end{equation}
\begin{equation}
\left<\tau_i\nu,f\right>=\left<\nu,\tau_i f\right>\label{eq:translation of measure}
\end{equation}
Particularly, we see $
\tau_i \eta_j =\eta_{i+j}, \left<\tau_i\nu,\eta_j\right> = \left<\nu,\eta_{i+j}\right>.
$

The second lemma says that any weak limit $\nu^*$ of the Ces\`{a}ro means of $\bar{\nu}$ under translation is a mixture of Bernoulli measures $\mu_\rho$, $0\leq \rho \leq 1$. \textcolor{black}{This is because $\nu^*$ is translation invariant and invariant with respect to the generator $\mathit{L}_0$ for AEP.}
\begin{lemma} \label{lm: Bernoulli measure}
	Let $\bar{\nu}$ be an invariant \textcolor{black}{measure} with respect to \textcolor{black}{the generator $L$}. For any weak limit $\nu^*$ of the Ces\`{a}ro means of $\bar{\nu}$ under translation:
	\begin{equation}\label{eq:Cesaro mean of nu}
	\nu^* = \lim_{k\to\infty}\nu^*_{n_k}=\lim_{k\to\infty} \frac{1}{n_k}\sum_{i=1}^{n_k}\tau_i\bar{\nu},
	\end{equation}
	it is \textcolor{black}{translation invariant} and invariant with respect to the generator $\mathit{L}_0$ for AEP. That is, for any local function $f$,
	\begin{align}
	\langle\nu^*,\tau_x f\rangle =&\langle \nu^*,f\rangle, \label{eq:translational invariant}
	\\
	\langle\nu^*,\mathit{L}_0 f \rangle=&0. \label{eq:invariant wrt L0}
	\end{align}
	where $L_0$ is translational invariant, and it acts on $f$ by 
	\begin{equation}\label{eq:generator for AEP}
	\mathit{L}_0 f (\eta) =\sum_{x,y\in \mathbb{Z}}p(y-x)\eta_x\left( 1-\eta_y\right)\left(f(\eta^{x,y})-f(\eta)\right)
	\end{equation}
	Particularly, there is a measure $w_{\rho}$ on $[0,1]$, such that 
	\begin{equation}\label{eq:T and I measure}
	\nu^* = \int \mu_\rho dw_{\rho}
	\end{equation} 
\end{lemma}
\begin{proof}
	\textcolor{black}{By Theorem VIII.3.9 \cite{Li85}, we only need to show translation invariance and invariance (\eqref{eq:translational invariant}, \eqref{eq:invariant wrt L0})} to get \eqref{eq:T and I measure}. The proofs for both are similar.
	
	For any local function $f$, \textcolor{black}{which is a bounded function on $\{0,1\}^\mathbb{Z}$ depending on finitely many $\xi_x$},
	\begin{align*}
	\langle\nu_{n_k}^*,\tau_1 f\rangle=& \frac{1}{n_k}\sum_{i=1}^{n_k}\langle\tau_i\bar{\nu},\tau_1 f\rangle\\
	=&\frac{1}{n_k}\sum_{i=1}^{n_k}\langle\tau_{i+1}\bar{\nu},f\rangle\\ 
	=& \langle\nu_{n_k}^*,f\rangle + O_f\left(\frac{1}{n_k}\right)
	\end{align*}
	
	Also, as $\bar{\nu}$ is invariant with respect to $\mathit{L}$ and $\mathit{L_0}\tau_i = \tau_i\mathit{L_0}$ ,we can compare \textcolor{black}{\eqref{eq: generator for Blocked ASEP}} with \eqref{eq:generator for AEP} and get,
	\begin{align*}
	\langle\nu_{n_k}^*,\mathit{L}_0 f\rangle=& \frac{1}{n_k}\sum_{i=1}^{n_k}\langle\tau_i\bar{\nu},\mathit{L}_0 f\rangle\\
	=&\frac{1}{n_k}\sum_{i=1}^{n_k}\langle\bar{\nu},\mathit{L}_0 (\tau_i f)\rangle \\	
	=&\frac{1}{n_k}\sum_{i=1}^{n_k}\langle\bar{\nu},\mathit{L}(\tau_i f)\rangle +\frac{1}{n_k}\sum_{i=1}^{n_k}\langle\bar{\nu},(\mathit{L}_0-\mathit{L})(\tau_i f)\rangle\\
	=& O_f\left(\frac{1}{n_k}\right).
	\end{align*}
	\textcolor{black}{In the last line, since $f$ is local, $(\mathit{L}_0-\mathit{L})(\tau_i f)$ is non-zero for finitely many $i$.}	Taking limits as $n_k \to \infty$, we get
	\eqref{eq:translational invariant} and \eqref{eq:T and I measure}.
\end{proof}


The third lemma says if an invariant measure $\bar{\nu}$ has a current with a zero mean and some weak limit $\nu^*$ of its Ces\`{a}ro means under translation is a Bernoulli measure $\mu_0$ with density 0, the densities of positive sites are identically $0$ for $\bar{v}$.
\begin{lemma}\label{lm: current 0 and asymptotic density}
	Let $\bar{v}$ be an invariant measure with respect to the generator $\mathit{L}$, and $\nu^*$ be a weak limit of of its Ces\`{a}ro means defined in \eqref{eq:Cesaro mean of nu}. If  $\langle\bar{\nu},C_{-1,1}\rangle=0$ and $\langle\nu^*,\eta_x\rangle =0$ for some $x$ (which implies for all $x$ since $\nu^*$ is \textcolor{black}{tranlation} invariant), we have $\langle\bar{\nu},\eta_x\rangle = 0$ for all $x>0$.
\end{lemma}
\begin{proof} We will divide the proof into 3 steps.
	\begin{enumerate}[label = S\arabic*.]
		\item Define a quantity \textcolor{black}{$G_i$}:	
		
		With identities $p(x,y) = p(y,x) + p(x,y)-p(y,x)$ and $\eta_x(1-\eta_y) -\eta_y(1-\eta_x) =\eta_x-\eta_y$, from \eqref{eq:current}, we get
		\begin{align*}
		\langle\bar{\nu},C_{i,i+1}\rangle
		=& \langle\bar{\nu},\sum_{x\leq i, i+1\leq y} p(y-x)(\eta_x-\eta_y)\rangle\\
		+& \langle\bar{\nu},\sum_{x\leq i, i+1\leq y} \textcolor{black}{ (p(y-x)-p(x-y)) \eta_y(1-\eta_x)}\rangle				
		\end{align*}
		Therefore, 	by Lemma \ref{lm: constant current for invariant measure}, we have, for \textcolor{black}{ $i \geq R$}, $\langle\bar{\nu},C_{i,i+1}\rangle =0$, and
		\begin{align}
		&\sum_{x\leq i, i+1\leq y} p(y-x)\langle\bar{\nu},\eta_y-\eta_x\rangle \notag \\ =& \sum_{x\leq i, i+1\leq y}  \textcolor{black}{(p(y-x)-p(x-y)) \langle\bar{\nu},\eta_y(1-\eta_x)}\rangle \label{eq:RHS} \end{align}
		\textcolor{black}{The choice for $i\geq R$ is to avoid $xy =0$. There is symmetry in the left hand side of \eqref{eq:RHS}, there are $2R$ terms with "odd" coefficients, $b_{-(j-1)} =-b_j, j =1,\dots,R$. We can write the sum as a difference of two shifted sums of $2R-1$ terms with even coefficients, $a_{-j}=a_j$, $j= 0,1,\dots,R-1$,
		\begin{align} 
	  \sum_{j=-(R-1)}^R b_j  \langle\bar{\nu},\eta_{i+j}\rangle
		=& \sum_{j=-(R-1)}^{R-1} a_j \langle\bar{\nu},\eta_{i+1+j}\rangle	- 	\sum_{j=-(R-1)}^{R-1} a_j \langle\bar{\nu},\eta_{i+j}\rangle \notag \\
		=:& G_{i+1}-G_i \label{eq:LHS}
		\end{align}
			} 
		where $G_i$ is defined for $i \geq R$ as
		\textcolor{black}{
		\begin{align}
		G_i :=& \sum_{j:\abs{j}\leq R-1} \sum_{k=\abs{j}+1}^{R} \left(k-\abs{j}\right)p(k)  \langle\bar{\nu},\eta_{i+j}\rangle  \notag \\
		=& A v_i, \label{eq:A >0,v>=0}
		\end{align}
	}
		and $A$ is a row vector with $2R-1$ positive entries \textcolor{black}{$a_j=\sum_{k=\abs{j}+1}^R\left(k-\abs{j}\right)p(k)$, for $\abs{j}\leq R-1 $}, and $v_i$ is a column vector with \textcolor{black}{$2R-1$ nonnegative entries $\langle\bar{\nu},\eta_{i+j}\rangle,$ for $\abs{j}\leq R-1$}.
		
		\item Convergence of $(G_i)_{i\geq R}$:
		
		\textcolor{black}{By the assumption $p(k)\geq p(-k)$ for $k>0$, we have the right hand side of \eqref{eq:RHS} is positive. Also, \eqref{eq:A >0,v>=0} implies that $G_i$ is bounded uniformly for $i\geq R$. Therefore, we get the convergence of $(G_i)_{i\geq R}$:
		\begin{equation} \label{eq:convergence of G_i}
		G_i \uparrow c, \text{ as } i \uparrow \infty.
		\end{equation}}
	
		\item From $\langle\nu^*,\eta_x\rangle=0$ to $\langle \bar{\nu},\eta_x\rangle=0$:
		
		\textcolor{black}{As the Ces\'{a}ro limit of a sequence is the same as its limit when both limits exist, by the definition \eqref{eq:Cesaro mean of nu} of $\nu^*$, \eqref{eq:A >0,v>=0}, and \eqref{eq:convergence of G_i}, we get $c=0$ from linearity. With strictly positive entries in $A$, we get, for $i\geq R$,}
		\[G_i = A v_i = 0,\]
		\textcolor{black}{and all entries in $v_i$ are $0$}. Particularly, $\langle\bar{\nu},\eta_{i+j}\rangle =0 $, for all indices $i+j$ with $i+j \geq R-(R-1) =1$.
	\end{enumerate} \end{proof}
	
	\textcolor{black}{We should notice that to write $G_i$ in forms of \eqref{eq:A >0,v>=0}, we need $i\geq R$. It is because we don't want terms involving $p(0,x)$ or $p(x,0)$. This condition holds for sites sufficiently right to the origin. We will see similar conditions in Theorem \ref{thm: right asymptotic density is 0} and Lemma \ref{lm: correpsondence} involved.}
	
	\subsection{Proof of Positive Currents in AEP with a Blockage}
	The coming theorem will be proved in Section \ref{sec: proof of u1/2 dominance}. It says, if the initial configuration has no particles after some point $x>0$, $\nu^*$ is dominated by $\mu_{\frac{1}{2}}$, \textcolor{black}{in the sense of \eqref{eq: zero right asymptotic density}.} \textcolor{black}{Let's recall from Section \ref{sec: Invariant Measure and Lower Bound} that the mean of empirical measures $\nu_{t}$ is defined by its action on local functions $\langle\nu_{t},f\rangle = \frac{1}{t}\mathbb{E}^{\nu_0,0}\left[\int_{0}^{t}f(\eta_s)\,ds\right]$} for some initial measure $\nu_0$.
	\begin{theorem}\label{thm: right asymptotic density is 0}
		Consider the AEP on lattice $\mathbb{Z}$ with 
		a blockage at the origin and $p(\cdot)$ has a positive mean $\sum z\cdot p(z)>0$. Let $\bar{\nu}$ be a weak limit of the mean of empirical measures $\bar{\nu}_{T_n}$, and $\nu^*$ \textcolor{black}{be defined via a subsequence mentioned in \eqref{eq:Cesaro mean of nu}}. \textcolor{black}{If there is an $x>R$ such that $\langle\nu_0,\eta_y\rangle =0$ for all $y \geq x$, we will have, for any finite set $A \subset \mathbb{Z}$},
		\begin{equation} \label{eq: zero right asymptotic density}
		\langle{\nu}^*,\prod_{x\in A}\eta_x\rangle \leq  \langle\mu_{\frac{1}{2}},\prod_{x\in A}\eta_x\rangle =2^{-\abs{A}}.
		\end{equation}
	\end{theorem}
	\begin{proof}
		\textcolor{black}{See Corollary \ref{cor: density estimate}}.
	\end{proof}
	
	Theorem \ref{thm:positive current when it is not fully bocked} is the main result of \textcolor{black}{Sections \ref{sec: Current Zero}, \ref{sec: ASEP on the half line with destruction}}. It says the current \textcolor{black}{through bond $(-1,1)$} is strictly positive for the AEP on $\mathbb{Z}$ when the initial measure is the step measure $\mu_{1,0}$. We will prove it by contradiction.
	\begin{theorem} \label{thm:positive current when it is not fully bocked}
		Suppose $p(\cdot,\cdot)$ satisfy assumptions at the beginning of this section. For the AEP on lattice $\mathbb{Z}$ with 
		a blockage at the origin, there is a lower bound  $C_1>0$ for the current \textcolor{black}{through bond $(-1,1)$}, 
		\begin{equation} \label{eq: lower bound for current}
		\liminf_{t\to \infty}\frac{1}{t}\mathbb{E}^{\mu_{1,0},0}\left[N_t\right]=\liminf_{t\to \infty}\left<\nu_t,C_{-1,1}\right>= C_1>0.\end{equation}
	\end{theorem}
	\begin{proof}
		\textcolor{black}{Let $N_t$ be the (net) number of particles jumping \textcolor{black}{through} bond $(-1,1)$ by time t, which is the same as \eqref{eq: Nt} when the tagged particle is not moving. Since no particles exist on the positive axis under the initial measure $\mu_{1,0}$, $N_t \geq 0$. As ${N_t - \int_{0}^{t}C_{-1,1}(\eta_s)\,ds }$ is a $\mathbb{P}^{\mu_{1,0},0}$- martingale (see Chapter 6.2 \cite{KLO}),
		 we get,} $\frac{1}{t}\mathbb{E}^{\mu_{1,0},0}\left[N_t \right] \geq 0$ for $t>0$, and $\liminf_{t\to \infty} \langle\nu_t,C_{-1,1}\rangle \geq 0$.
		
		Suppose $C_1=0$, by tightness, there is an invariant measure $\bar{\nu}$ with a zero current	$\langle \bar{\nu},C_{-1,1} \rangle=0$. By Lemma \ref{lm: constant current for invariant measure}, $\langle \bar{\nu},C_{x,x+1} \rangle=0$, for $x\geq R$. We have
		\begin{align*}
		\langle \nu^*_{n_k},C_{R,R+1}\rangle &=\frac{1}{n_k}\sum_{i=1}^{n_k}\langle\tau_i\bar{\nu},C_{R,R+1}\rangle  \notag \\
		&= \frac{1}{n_k}\sum_{i=1}^{n_k}\langle\bar{\nu},C_{R+i,R+i+1}\rangle = 0.
		\end{align*} Then, for any weak limit $\nu^* = \lim_{k\to\infty} \nu^*_{n_k} =\lim_{k\to\infty} \frac{1}{n_k}\sum_{i=1}^{n_k}\tau_i\bar{\nu},$ 
		\begin{equation}	
		\langle \nu^*,C_{R,R+1}\rangle=0.\end{equation}

		On the other hand, by Lemma \ref{lm: Bernoulli measure}, $\nu^*$ is a mixture of Bernoulli measures, that is, $\nu^* = \int \mu_\rho dw_{\rho}$. A computation shows $\langle\mu_{\rho},C_{R,R+1}\rangle = \rho(1-\rho) \sum_{i\leq R, j\geq R+1} \textcolor{black}{ \left( p(i,j)-p(j,i)\right)}$, which is strictly positive unless $\rho = 0$ or $1$.  As a consequence, 
		\begin{equation}
		\nu^* =	w_1\mu_1 + w_0\mu_0, 
		\end{equation}
		with $w_1+w_0=1.$
		
		By Theorem \ref{thm: right asymptotic density is 0}, we have $w_1 \leq 2^{-\abs{A}}$, for any finite set $A \subset \mathbb{Z}$. This implies $ w_1=0$ and $w_0=1$. Then, by Lemma \ref{lm: current 0 and asymptotic density}, we have, for $x>0$,
		\begin{equation}\label{eq:right profile}
		\langle\nu^*,\eta_x\rangle = 0, \text{ and } \langle\bar{\nu},\eta_x\rangle = 0. 
		\end{equation}
		By the particle-hole duality, \textcolor{black}{i.e. viewing holes as particles and reversing the lattice $\mathbb{Z}$, we get a result like \eqref{eq:right profile}}: for $x<0$,
		\begin{equation}\label{eq:left profile}
		\langle\bar{\nu},\eta_x\rangle = 1.
		\end{equation}
		 \eqref{eq:right profile} and \eqref{eq:left profile} imply the current $\langle \bar{\nu},C_{-1,1} \rangle$ is positive, which is a contradiction.
	\end{proof}
	
	\section{AEP on Half Line with Creation and Annihilation} \label{sec: ASEP on the half line with destruction}
	
	To show Theorem \ref{thm: right asymptotic density is 0}, we will consider an auxiliary process: the AEP on the half line with creation and annihilation. This model has a long history and was studied by Liggett in \cite{Li75} and \cite{Li77}. We will use some results from \cite{Li75} and \cite{Li77} to show the estimate \eqref{eq: zero right asymptotic density} in Theorem \ref{thm: right asymptotic density is 0}. 
	
	\subsection{Comparison between AEP on Half Line with Creation and AEP with a Blockage}
	We first describe the AEP on the half line with only creation formally as follows. Particles \textcolor{black}{move according to} asymmetric exclusion process on half line $[1,\infty)$ with jump rates $p(x,y) = p(y-x)$. If a positive site $y>0$ is vacant, a particle is created at $y$ with rates $\sum_{x\leq 0} p(y-x)$. Also, no particles are allowed to jump out of the positive half line. Clearly, if we consider the AEP on $\mathbb{Z}$ with an immediate creation of particles on $(-\infty,0]$ when sites are vacant, we get exactly the same dynamic on positive sites.   
	
	The first lemma connects  the AEP with a blockage at a site with the AEP on the half line with creation. Denote the AEP with a blockage at site 0 by $\eta_t$, which has a probability measure $P$; denote the AEP on the half line with creation by $\zeta_t$, which has a probability measure $Q$. 
	
	\begin{lemma} \label{lm: correpsondence}
		Suppose AEP with a blockage at a site starts from the initial measure $\mu_{1,0}$ and the AEP on the half line with creation starts from the Bernoulli measure $\mu_0$ on positive axis. Then, for any finite subset $A \subset \mathbb{Z}_+$, and any $t \geq 0$,
		\begin{equation}
		\label{eq:comparison}
		P(\eta_t(x + R) = 1, \text{for all }x\in A) \leq Q(\zeta_t(x)= 1, \text{for all }x\in A).
		\end{equation}
	\end{lemma}
	\begin{proof} $R$ is the range of jump rates $p(\cdot)$ as defined at the beginning of Section \ref{sec: Current Zero}. \textcolor{black}{We use it to avoid sites too close to the orgin.} We can view holes and particles in the AEP with blockage as three classes of particles, and we use 1,2, and 3 to denote each class:
		\begin{enumerate}[label=\alph*.]
			\item Particles are always labeled as first class particles.
			\item A hole becomes a second class particle whenever it visits or starts from a site on $(-\infty,R]$. 
			\item A hole which never visits sites on $(-\infty,R]$ is labeled as a third class particle. 
		\end{enumerate}
		\textcolor{black}{Particularly, a second class particle can jump to a third class particle because holes can interchange their positions in AEP, but the third class cannot jump to a second class particle. It is easy to see, jumps of particles with lower numbers to sites occupied by particles with higher numbers are allowed, while jumps of particles with higher numbers to sites occupied by particles with lower numbers are not allowed. }
		Denote this process with three classes of particles by $\xi_t$, which has a probability measure $\tilde{P}$.

		See Figure \ref{Figure 5} for an example. In this example, \textcolor{black}{$\xi_0$ is the initial configuration; $\xi_{t_1}$ is the configuration after a (first class) particle jumps from $-1$ to $1$, a (second class) particle jumps from $2$ to $3$, and a (second class) particle jumps from $4$ to $6$; $\xi_{t_2}$ is a configuration at a general time $t_2$.}
		
		\begin{figure}[h!]
			\begin{center}
				\begin{tikzpicture} 
				\draw[black, thick] (-4.3,0) -- (4.3,0)node[black,right] {$\xi_0$};  
				\draw [black,fill] (-0.1-2,0.1) rectangle (0.1-2,-0.1) node [black,below=4] {Blockage}; 
				
				\foreach \x in {1,3,4}
				\draw (\x,0) circle (0.1)
				node [black,below=4] {$3$};
				
				\foreach \x in {-1,0,2}
				\draw (\x,0) circle (0.1)
				node [black,below=4] {$2$};
				
				\draw (0,0) circle (0)
				node [black,above=4] {$R=2$};
				
				\foreach \x in {-4,-3}
				\draw [black,fill] (\x,0) circle (0.1)
				node [black,below=4] {$1$};

				
				\draw[black, thick] (-4.3,-1) -- (4.3,-1) node[black,right] {$\xi_{t_1}$};  
				\draw [black,fill] (-0.1-2,0.1-1) rectangle (0.1-2,-0.1-1) node [black,below=4] {Blockage}; 
				
				\foreach \x in {2,3}
				\draw (\x,0-1) circle (0.1)
				node [black,below=4] {$3$};
				
				\foreach \x in {1,4}
				\draw [black,fill] (-\x,0-1) circle (0.1)
				node [black,below=4] {$1$};
				
				\foreach \x in {-3,0,1,4}
				\draw (\x,0-1) circle (0.1)
				node [black,below=4] {$2$};
				
				\draw[black, thick] (-4.3,-2) -- (4.3,-2) node[black,right] {$\xi_{t_2}$};  
				\draw [black,fill] (-2-0.1,0.1-2) rectangle (0.1-2,-0.1-2) node [black,below=4] {Blockage}; 
				
				\foreach \x in {1}
				\draw (\x,0-2) circle (0.1)
				node [black,below=4] {$3$};
				
				\foreach \x in {-1,0,2,-3,4}
				\draw [black,fill] (\x,0-2) circle (0.1)
				node [black,below=4] {$1$};
				
				\foreach \x in {-4,3}
				\draw (\x,0-2) circle (0.1)
				node [black,below=4] {$2$};
				
				\end{tikzpicture}
			\end{center}
			\caption{The AEP with Blockage and 3 Classes of Particles}
			\label{Figure 5}
		\end{figure}
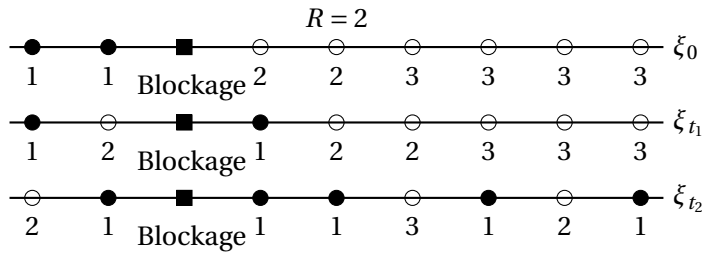
		
		Identifying the first class particles in $\tilde{P}$ as standard particles in $\mathit{P}$, we have
		\begin{equation} \label{eq: identifying three classes of particles}
		\tilde{P}(\xi_t(x) = 1, \text{for all }x\in A+R) = P(\eta_t(x) = 1, \text{for all }x\in A+R)
		\end{equation}
		On the other hand, \textcolor{black}{the dynamic of the third class particles correspond to the dynamic of holes in $Q$ after translation by $+R$. Therefore}, we have for any subset $B\subset A$,
		\begin{equation} \label{eq: identifying P, Q process}
		\tilde{P}(\xi_t(x) \neq 3, \text{for all }x\in B+R) = Q(\zeta_t(x) \neq 0, \text{for all }x\in B)
		\end{equation}
		As a consequence of \eqref{eq: identifying three classes of particles} and \eqref{eq: identifying P, Q process},
		\begin{align*}
		P(\eta_t(x+R) = 1, \text{for all }x\in A+R) \leq& \tilde{P}(\xi_t(x+R) =1 \text{ or }2, \text{for all }x\in A)
		\\		
		=& Q(\zeta_t(x)= 1, \text{for all }x\in A). 
		\end{align*}
	\end{proof}
	
	\subsection{Couplings in the AEP with Creation and Annihilation}
	By the above lemma, we can study the asymptotic behavior of the AEP on half line with only creation. The main theorem of this section is Theorem \ref{thm:main convergence for Q process}. The proof of Theorem \ref{thm:main convergence for Q process} can be derived from results in \cite{Li77}, with stochastic orderings (couplings). We start with some notion and results from \cite{Li77} and \cite{Li75}.
	
	Consider a subset $D_{m,n} = \{m,m+1,\dots,n\} \subset \mathbb{Z}$, the configuration space on $D_{m,n}$  \textcolor{black}{is} $\mathbb{X}_{m,n} = \{0,1\}^{D_{m,n}}$, and a probability measure $v_{m,n}$ on $\mathbb{X}_{m,n}$. We can extend $v_{m,n}$ to a measure on $\mathbb{X}_{-\infty,\infty} = \{0,1\}^{\mathbb{Z}}$ by taking product measure. Let $\lambda, \rho \in [0,1]$, we can have
	\textcolor{black}{
	\begin{align} 
	v_{m,n;\lambda,\rho} =& \mu^{-\infty,m-1}_\lambda \otimes v_{m,n} \otimes \mu^{n+1,\infty}_\rho,\label{eq: extend measures1 }\\
	v_{m,\infty;\lambda} =& \mu^{-\infty,m-1}_\lambda \otimes v_{m,\infty} \label{eq: extend measures2 } \end{align}
}
	where \textcolor{black}{$\mu^{-\infty,m-1}_\lambda$ is a Bernoulli measure with density $\lambda$ on $\mathbb{X}_{-\infty,m-1} = \{0,1\}^{\{i:i<m\}}$ and $\mu^{n+1,\infty}_\rho$ is a Bernoulli measure with density $\rho$ on $\mathbb{X}_{n+1,\infty} = \{0,1\}^{\{i:i>n\}}$. With this extension, we can compare measures on different $\mathbb{X}_{m,n}$ with partial orders on the space of measures on $\mathbb{X}_{-\infty,\infty}$.}
	
	We first define partial orders on the space of configurations $\mathbb{X}_{-\infty,\infty}$   
	\begin{equation} \label{eq:basic coupling, for config}
	\eta \geq \xi \Leftrightarrow \eta(x) \geq \xi(x) \text{ for all } x \in \mathbb{Z}
	\end{equation}
	and then we can define partial orders on the space of probability measures via stochastic ordering. For any \textcolor{black}{increasing} local function $f$ (with respect to  \eqref{eq:basic coupling, for config}),
	\begin{equation}\label{eq:basic coupling, stoch ordering}
	\nu \geq \mu \Leftrightarrow \langle \nu,f\rangle \geq \langle \mu,f\rangle
	\end{equation}
	
	We will consider the AEP with creation and annihilation on both a finite system and an infinite system. The former is a process on $\mathbb{X}_{m,n}$ with a generator $\Omega_{m,n}^{\lambda,\rho}$. $\Omega_{m,n}^{\lambda,\rho}$ acts on a local function $f$ by
	\begin{align}
	\Omega_{m,n}^{\lambda,\rho}f(\eta) =& \sum_{x<m, y\in D_{m,n}} \left( p(x,y)\lambda (1-\eta_y) +  p(y,x)(1-\lambda)\eta_y \right) \left(f(\eta^y) -f(\eta)\right) \notag \\
	&+\sum_{x\in D_{m,n},y>n} \left( p(x,y)\eta_x (1-\rho) +  p(y,x)\rho(1-\eta_x) \right) \left(f(\eta^x) -f(\eta)\right) \notag \\
	&+\sum_{x,y\in D_{m,n}}  p(x,y)\eta_x(1-\eta_y) \left(f(\eta^{x,y}) -f(\eta)\right),
	\end{align}
	where \textcolor{black}{
	\[
	\eta^x(z) = \begin{cases}
	1-\eta(x)  &, \text{ if }  z = x\\
	\eta(z)   &,  \text{ else }.
	\end{cases}  \]}
	
	And the later is a process on $\mathbb{X}_{m,\infty}$ with a generator $\Omega_{m,\infty}^{\lambda}$. $\Omega_{m,\infty}^{\lambda}$ acts on a local function $f$ by
	\begin{align}
	\Omega_{m,\infty}^{\lambda}f(\eta) =& \sum_{x<m, y\geq m} \left( p(x,y)\lambda (1-\eta_y) +  p(y,x)(1-\lambda)\eta_y \right) \left(f(\eta^y) -f(\eta)\right) \notag \\
	&+\sum_{x,y\geq m }  p(x,y)\eta_x(1-\eta_y) \left(f(\eta^{x,y}) -f(\eta)\right).
	\end{align}
	
%
	\subsection{Liggett's Results and Their Consequences }\label{sec: proof of u1/2 dominance}
	Below are results from \cite{Li75} and \cite{Li77}. \textcolor{black}{Particularly, the monotonicity in the first part of Lemma \ref{lm: correspondence btw Li 77} guarantees interchanging of limits. Recall $\mu_\rho^{m,n}$ is a Bernoulli measure on $\mathbb{X}_{m,n}$ with density $\rho$.}
	\begin{lemma}\label{lm: correspondence btw Li 77}
		\textcolor{black}{Assume} $1\geq \lambda \geq \rho \geq 0$, and  $m\leq n \leq \infty$. Let \textcolor{black}{$\nu_{m,n}^{\lambda,\rho}(t):= \mu^{-\infty,m-1}_\lambda \otimes  \left(\mu_\rho^{m,n} S_{m,n}^{\lambda,\rho}(t)\right)\otimes \mu^{n+1,\infty}_\rho$}. Then we have,
		\begin{enumerate}
			\item In the sense of \eqref{eq:basic coupling, stoch ordering}, the probability measure \textcolor{black}{$\nu_{m,n}^{\lambda,\rho}(t)$} is increasing in parameters $m,n,t, \lambda$ and $\rho$. 
			\item \textcolor{black}{Let $\bar{\nu}_{m,n;\lambda,\rho} = \lim_{t\uparrow\infty} \nu_{m,n}^{\lambda,\rho}(t) $. $\bar{\nu}_{m,n;\lambda,\rho}$ converges to a unique limit $\bar{\nu}_{m;\lambda,\rho}$ as $n$ goes to $\infty$. And $\bar{\nu}_{m;\lambda,\rho} = \lim_{t\uparrow\infty}\mu^{-\infty,m-1}_\lambda\otimes \left(\mu_{\rho}^{m,\infty} S_{m,\infty}^{\lambda}(t)\right) $.}
			\item For $n-m> 2R$, the current in $D_{m,n}$ has two lower bounds:
			\begin{equation} \label{eq: lower bound for current2}
			\langle \bar{\nu}_{m,n;\lambda,\rho},C_{x,x+1} \rangle \geq w \cdot \max\{\lambda(1-\lambda),\rho(1-\rho)\}
			\end{equation}
			where $w=\sum_{\abs{k}\leq R} kp(k)$.
		\end{enumerate} 
	\end{lemma}
	\begin{proof}
		The first part \textcolor{black}{of Lemma \ref{lm: correspondence btw Li 77}} is proved in Theorem 2.4, 2.13 in \cite{Li75}. \textcolor{black}{The second part is a consequence of the monotonicity in parameters from the first part and the Trotter  Theorem, see Proposition 2.2 in \cite{Li75}. We only show the last equality:
			\begin{align*}
				\lim_{t\uparrow\infty}\mu_\lambda^{-\infty,m-1}\otimes \left(\mu_{\rho}^{m,\infty} S_{m,\infty}^{\lambda}(t)\right) =& \lim_{t\uparrow\infty}\mu_\lambda^{-\infty,m-1}\otimes  \left(\lim_{n\uparrow \infty}\left( \mu_{\rho}^{m,n} S_{m,n}^{\lambda,\rho}(t)\right)\otimes \mu_\rho^{n+1,\infty}\right) \notag \\
				=& \lim_{t\uparrow\infty}\lim_{n\uparrow \infty}\mu_\lambda^{-\infty,m-1}\otimes \left( \mu_{\rho}^{m,n} S_{m,n}^{\lambda,\rho}(t)\right)\otimes \mu_\rho^{n+1,\infty} \notag \\
				=& \lim_{t\uparrow\infty}\lim_{n\uparrow \infty} \nu_{m,n}^{\lambda,\rho}(t) \notag\\
				=&\lim_{n\uparrow \infty}\lim_{t\uparrow\infty} \nu_{m,n}^{\lambda,\rho}(t) = \bar{\nu}_{m;\lambda,\rho}.
			\end{align*}
			Particularly, the first line is by Proposition 2.2 \cite{Li75}, and the interchanging of limits in the third line is by the monotonicity in parameters $n,t$ from the first part.} The third part \textcolor{black}{of Lemma \ref{lm: correspondence btw Li 77}} is by the proof of Propositin 2.6 in \cite{Li77}. It is a consequence of the monotonicity of $\bar{v}_{m,n,\lambda,\rho}$ in $m,n$ and a direct computation of currents at two boundaries $C_{m-1,m}$ and $C_{n,n+1}$. 
	\end{proof}
	
	The main theorem of this section says the AEP on half line with creation has a limiting measure. When translated along the positive direction, the limiting measure converges to the Bernoulli measure $\mu_{\frac{1}{2}}$ in the Ces\`{a}ro sense. \textcolor{black}{This corresponds to the limiting measure of usual AEP is the Bernoulli measure $\mu_\frac{1}{2}$ when the initial measure is the step measure $\mu_{1,0}$.}
	\begin{theorem} \label{thm:main convergence for Q process}Assume the AEP on half line with creation has the initial configuration with only holes in positive sites. Let $m_t$ be measures on $\{0,1\}^{\mathbb{Z}_+}$ with $\langle m_t,\prod_{x\in A}\eta_x \rangle = Q(\zeta_t(x)=1,\text{for all }x\in A)$ for any finite subset $A \subset \mathbb{Z}_+$. 
		Then we have \textcolor{black}{the following},
		\begin{align}
		&\lim_{t\to\infty}m_t = \bar{m} \text{ exists} \label{eq: limiting measure}\\
		&\lim_{N\to \infty}\frac{1}{N}\sum_{i=1}^{N}\langle \bar{m},\prod_{x\in A+i}\eta_x \rangle = 2^{-\abs{A}}.   \label{eq:key estimates for Q} 
		\end{align}
		
	\end{theorem}
	\begin{proof}
		\textcolor{black}{Assume $\lambda \geq \rho$, from Lemma \ref{lm: correspondence btw Li 77}, $\bar{\nu}_{m;\lambda,\rho}$ is the limiting measure of $\nu_{m,n}^{\lambda,\rho}(t)$ as $t,n$ go to $\infty$.} It is also increasing in $m,\lambda,\rho$. Therefore, we can define a unique limiting measure $m(\lambda,\rho)$, which is also increasing in $\lambda$ and $\rho$,
		\begin{equation} 
		m(\lambda,\rho) =\lim_{m\to \infty}\bar{\nu}_{-m;\lambda,\rho}.
		\end{equation}
		It is also the same as the limit of the Ces\`{a}ro means of $\bar{v}_{m;\lambda,\rho}$ under translation:
		\begin{align*}\tau_i \bar{\nu}_{m;\lambda,\rho} =& \lim_{n\to\infty}\tau_i\bar{\nu}_{m,n;\lambda,\rho} = \lim_{n\to\infty}\bar{\nu}_{m-i,n-i;\lambda,\rho} =  \bar{\nu}_{m-i;\lambda,\rho}(\lambda,\rho),\\
		\textcolor{black}{m(\lambda,\rho)} =&\lim_{N\to\infty}\tau_N \bar{\nu}_{m;\lambda,\rho} = \lim_{N\to\infty}\frac{1}{N}\sum_{i=1}^{N}\tau_i  \bar{\nu}_{m;\lambda,\rho}.
		\end{align*}
		
		\textcolor{black}{On the other hand, we see that the limit $m(\lambda,\rho)$ is translation invariant and invariant with respect to the $\mathit{L}_0$, following similar computations in Lemma \ref{lm: Bernoulli measure}. Therefore, $m(\lambda,\rho)$ is a mixture of Bernoulli measures. For any Bernoulli measure $\mu_\rho$, a small computation shows $\langle \mu_\rho , C_{R,R+1}\rangle = w \rho(1-\rho) \leq  \frac{1}{4} w$, where $w = \sum_{\abs{k} \leq R} kp(k)$. As a consequence, we get an upper bound, for any $\lambda \geq \rho$,
		\begin{equation}\label{eq:upper bound for current, wrt u^*}
		\langle m(\lambda,\rho) , C_{R,R+1}\rangle \leq \frac{1}{4} w,
		\end{equation} and equality holds if and only if  $m(\lambda,\rho) =\mu_{\frac{1}{2}}$.}

		\textcolor{black}{The lower bound \eqref{eq: lower bound for current2} in Lemma \ref{lm: correspondence btw Li 77} indicates
		$\langle m\left(\frac{1}{2},0\right) , C_{R,R+1}\rangle \geq \frac{1}{4} w$, $\langle \textcolor{black}{m\left(1,\frac{1}{2}\right) }, C_{R,R+1}\rangle \geq \frac{1}{4}w$. We see that $m\left(\frac{1}{2},0\right)$ and $m\left(0,\frac{1}{2}\right)$ are  Bernoulli measures with the same density $\frac{1}{2}$, 
		\[m\left(\frac{1}{2},0\right)  =  m\left(1,\frac{1}{2}\right)  = \mu_{\frac{1}{2}}.\]}
		Together with monotonicity in $\lambda,\rho$ , we get for $\lambda \geq \frac{1}{2} \geq \rho$,
		\begin{equation}\textcolor{black}{
		\mu_{\frac{1}{2}}=m\left(\frac{1}{2},0\right) \leq m(\lambda,\rho)\leq m\left(1,\frac{1}{2}\right) =\mu_{\frac{1}{2}}.} \end{equation}
		
		We can conclude the proof by letting $\lambda =1, \rho =0$, and identifying 
		$m_t$ as the restriction of\textcolor{black}{ $\nu_{0,\infty}^{1,0}(t)$ on $\mathbb{X}_{0,\infty}$. Taking weak limits (again by Lemma \ref{lm: correspondence btw Li 77})}, we  get \eqref{eq:key estimates for Q}
		\[\lim_{N\to \infty}\frac{1}{N}\sum_{i=1}^{N}\langle \bar{m},\prod_{x\in A+i}\eta_x \rangle =\langle m(1,0),\prod_{x\in A}\eta_x \rangle = 2^{-\abs{A}}. \]  	
	\end{proof}
	
	\textcolor{black}{We give the proof of Theorem \ref{thm: right asymptotic density is 0} as a corollary of Theorem \ref{thm:main convergence for Q process}}.
	\begin{corollary}\textcolor{black}{(proof of Theorem \ref{thm: right asymptotic density is 0})}\label{cor: density estimate}
		Let $\bar{\nu}$ be a weak limit of the mean of empirical measures $\bar{\nu}_{T_n}$, and $\nu^*$ be a weak limit of the Ces\`{a}ro means of $\bar{\nu}$ under translation \eqref{eq:Cesaro mean of nu}. Then for any finite set $A \subset \mathbb{Z}$,
		\begin{equation} 
		\langle{\nu}^*,\prod_{x\in A}\eta_x\rangle \leq 2^{-\abs{A}}.
		\end{equation}
	\end{corollary}
	
	\begin{proof}
		Consider some weak limit $\bar{\nu}$ of the means of the empirical measure for the P-process defined by \eqref{eq:Cesaro mean of nu}. By \eqref{eq:comparison} and \eqref{eq: limiting measure}, we have, for any $A \subset \mathbb{Z}_+$
		\[\langle \bar{\nu},\prod_{x\in A}\eta_{x+R}\rangle \leq \langle \bar{m}_,\prod_{x\in A}\eta_x\rangle  \] 
		Therefore, for $i>0$,
		\[\langle \tau_i \bar{\nu},\prod_{x\in A}\eta_{x+R}\rangle = \langle \bar{\nu},\prod_{x\in A}\eta_{x+R+i}\rangle \leq \langle \bar{m}_,\prod_{x\in A}\eta_{x+i}\rangle\]
		Therefore, by \eqref{eq:key estimates for Q} and \eqref{eq:Cesaro mean of nu},
		\[
		\langle  \nu^*,\prod_{x\in A}\eta_{x+R}\rangle \leq \lim_{N_k\to \infty}\frac{1}{N_k}\sum_{i=1}^{N_k}\langle \bar{m},\prod_{x\in A}\eta_{x+i} \rangle = 2^{-\abs{A}}.  \]
		We extend the inequality to any subset $A$ of $\mathbb{Z}$ since $v^*$ is translational invariant by  Lemma \ref{lm: Bernoulli measure}.
	\end{proof}
	
		\section{Proofs of Theorem \ref{Thm:ballistic behavior for left moving tagged particle} and Theorem \ref{Thm:positive speed for slow moving tagged particle}} \label{sec:proof of main}	
		In this section, we prove \textcolor{black}{Theorem \ref{Thm:ballistic behavior for left moving tagged particle} and Theorem \ref{Thm:positive speed for slow moving tagged particle}. Let's start with the proof Theorem \ref{Thm:positive speed for slow moving tagged particle}, and we will see the proof of Theorem \ref{Thm:ballistic behavior for left moving tagged particle} follow similar arguments.}
		

		%
		%
		
		\begin{proof}(Theorem \ref{Thm:positive speed for slow moving tagged particle})
			We divide the proof into two steps.
			\begin{enumerate}[label = Step\arabic*.]
				\item Existence of $q(\cdot)$ and ergodic measure $\nu_e$ for the environment process $\xi_t$:
				
				By Theorem \ref{thm:positive current when it is not fully bocked}, we can define $C_1:= \liminf_{t\to \infty}\frac{1}{t}\mathbb{E}^{\mu_{1,0},0}\left[N_t\right] >0 $. Then by Theorem \ref{thm:replacing estimate}, for any nearest-neighbor $q(\cdot)$, we have \textcolor{black}{$C_0 := C_1 - (q(1)+q(-1))$}, such that
				\[\liminf_{t\to \infty}\frac{1}{t}\mathbb{E}^{\mu_{1,0},q}\left[N_t\right]\geq C_0.\]
				As a consequence, by Lemma \ref{prop:displacement and spped of tagged}, there is an invariant measure $\bar{\nu}$ for the environment process $\xi_t$, such that 
				\begin{align} \label{key bound}  
				\liminf_{t\to \infty}\mathbb{E}^{\mu_{1,0},q}\left[\frac{D_{t}}{t}\right] = \langle \bar{\nu}, f \rangle \geq  \frac{q(1)}{p(2)}C_0 - (q(-1)-q(1)), 				\end{align} where 
				\begin{equation*}
					f(\xi) = q(-1)(1-\xi_1) - q(1)(1-\xi_{-1}).
				\end{equation*} We can choose $q(-1) > q(1)$, to obtain a strict positive lower bound for \eqref{key bound}.
				
				On the other hand, \textcolor{black}{the set of invariant measures is a nonempty closed convex compact set by tightness. There exists an invariant measures satisfying \eqref{key bound}, which is a convex combination of two extremal points. Therefore, there is an extremal point $\nu_e$ of the set of invariant measures, such that, $\nu_e$ is ergodic for $\xi_t$, and $\nu_e$ satisfies \eqref{key bound}
				\begin{equation}\langle \nu_e, f \rangle \geq \frac{q(1)}{p(2)}C_0 - (q(-1)-q(1)) >0. \label{eq: last lower bound estimate2}\end{equation}}
				
				\item The speed of the tagged particle is positive:
				
				We can use $\mathbb{P}^{\nu_e,p}-$ martingales, (see Chapter 6.2 \cite{KLO}) 
				\[M_t = D_t- \int_{0}^{t}f(\xi_s)\,ds,\]
				where $M_t$ is a martingale with quadratic variance of order $t$.
				As $\nu_e$ is invariant and ergodic for the environment process $\xi_t$, we apply Ergodic Theorem, and get 
				\begin{equation}
				\lim_{t\to\infty} \frac{D_t}{t} = \langle \nu_e,f \rangle >0,\quad  \mathbb{P}^{\nu_e,q} -a.s 
				\end{equation} 
				
			\end{enumerate}
		\end{proof}	
		
	\textcolor{black}{In the case when the tagged particle only has pure left jumps, following arguments in Step 2 of the above proof, we only need to show that $\limsup_{t\to\infty}\frac{1}{t}\int_{0}^{t}f(\xi_s)\,ds \leq c$ for some $c<0$. By couplings introduced in Section \ref{sec:coupling}, Theorem \ref{thm:positive current when it is not fully bocked} and Remark \ref{rm: similar proofs}.2, we can obtain this bound.}
	
	\begin{proof}(Theorem \ref{Thm:ballistic behavior for left moving tagged particle})
			Following the proof of Theorem \ref{thm:replacing estimate}, we have a coupling 
			\[\vec{X}_t=(\vec{X}_0,\tilde{L}_R,p,q) \succeq ( \vec{X}_0,\tilde{L},p,0) = \vec{Y}_t, \]
			which implies under some joint distribution, we have
			\eqref{ineq:index change1}
			\[
				F(\vec{X}_0)- F(\vec{X}_t) \geq F(\vec{Y_0})-F(\vec{Y}_t),  \text{a.s.} \] 
			As the tagged particle only jumps towards left, \textcolor{black}{$r_{\vec{X}}(t)=0$. By \eqref{ineq:index change1}, \eqref{eq:change in index_x}, and  \eqref{eq:change in index_y},} we get
				\begin{equation*}
				N_{\vec{X}}(t) \geq N_{\vec{Y}}(t) \quad \text{ a.s.} \end{equation*} 
				Then by Remark \ref{rm: similar proofs}.2 and Theorem \ref{thm:positive current when it is not fully bocked}, we have the following estimate
					\begin{equation} \label{key estimate}
					\liminf_{t\to\infty}	\frac{1}{t}N_{\vec{X}}(t) \geq \liminf_{t\to\infty}	\frac{1}{t}N_{\vec{Y}}(t) = C_1, \text{a.s.} \end{equation}
			
			On the other hand, for the environment process of AEP with a driven tagged particle, we compute $L\eta_{-1} + C_{-1,1}$ by \eqref{Eq:generator} and \eqref{eq:current}:
			\begin{align}
				L\eta_{-1} + C_{-1,1} =& \left(p(2) \eta_{-3}(1-\eta_{-1}) + p(1)\eta_{-2}(1-\eta_{-1})-p(2)\eta_{-1}(1-\eta_1) +q(-1)(1-\eta_{-1})\right) \notag \\
				&+ \left( p(2)\eta_{-1}(1-\eta_1) - p(-2)\eta_{1}(1-\eta_{-1}) \right) \notag\\
				=& (1-\eta_{-1})\left(p(2)\eta_{-3}+p(1)\eta_{-2}+q(-1) -p(-2)\eta_{-1} \right)  \notag \\
				\leq& (1-\eta_{-1})\cdot C_4= - \frac{C_4}{q(-1)} f(\xi), \label{eq:rewrite compensentor}
			\end{align}
			where $C_4= p(2)+p(1)+q(-1)$, and $f(\xi) = -q(-1)(1-\xi_{-1})$.
			
			We can use three $\mathbb{P}^{\mu_{1,0},p}-$ martingales, (see Chapter 6.2 \cite{KLO})
			\begin{align*}
			 \quad N_t - \int_{0}^{t} C_{-1,1}(\xi_s)\,ds ,\quad 
				\xi_t(-1) -\int_{0}^{t} L\eta_{-1}(\xi_s)\,ds,\quad 	D_t - \int_{0}^{t} f(\xi_s)\,ds,
			\end{align*}
			which all have quadratic variance of order $t$. Dividing $t$ and taking limits, we see from \eqref{key estimate} and $\abs{\xi_t(-1)} \leq 1$ that, $\mathbb{P}^{\mu_{1,0},q}$-a.s.,
			\begin{align*}
				\liminf_{t\to \infty}\frac{1}{t}\int_{0}^{t} C_{-1,1}(\xi_s)\,ds &\geq \liminf_{t\to \infty}\frac{1}{t} N_{\vec{X}}(t) \geq C_1,\\
				\lim_{t\to \infty}\frac{1}{t}\int_{0}^{t} L\eta_{-1}(\xi_s)\,ds &= \lim_{t\to \infty}\frac{1}{t} \left(\xi_t(-1)-\xi_0(-1)\right) =0.
			\end{align*} By \eqref{eq:rewrite compensentor}, we get
			\[ \limsup_{t\to \infty}\frac{1}{t}\int_{0}^{t} f(\xi_s)\,ds \leq - C_1 \frac{q(-1)}{q(-1)+p(1)+p(2)},\quad  \mathbb{P}^{\mu_{1,0},q}-\text{a.s.} \]
			Choosing $c:= -C_1 \frac{q(-1)}{q(-1)+p(1)+p(2)}<0$, we obtain
			\[ \limsup_{t\to \infty} \frac{D_t}{t} =\limsup_{t\to \infty}\frac{1}{t}\int_{0}^{t} f(\xi_s)\,ds \leq c <0 ,\quad  \textcolor{black}{\mathbb{P}^{\mu_{1,0},q}}\--a.s. \]
	\end{proof}
	
	\section{Ballistic Behavior of a Fast Tagged Particle in AEP \label{sec:Lower bounds for a fast tagged particle}}
	
	In this section, we will prove Theorem \ref{Thm:ballistic behavior for fast moving tagged particle}. In this case, both the green tagged particle and red particles have non-nearest-neighbor jump rates and the means of jump rates are positive. This is a scenario different from Theorem \ref{Thm:positive speed for slow moving tagged particle}. Particularly, the jump rate $\beta = \sum_z q(z)$ of the tagged particle can be larger than the jump rate $\lambda = \sum_z p(z)$ of red particles. 
	
	We briefly discuss the steps of the proof. \textcolor{black}{We will modify the auxiliary process introduced in Section \ref{sec:coupling}. Instead of labeling red particles and considering their positions relative to the tagged particle, we will also label the tagged particle, and keep track of its label (see \eqref{eq:modified auxiliary process}). With this modified auxiliary process, we can couple the ordered particles (including the tagged particle) in the AEP with the ordered particles in the usual AEP. By investigating the change in labels of tagged particles in both processes, we can compare their positions. We obtain a lower bound for the driven tagged particle in AEP by estimates from the usual AEP}.
	
	\subsection{Assumptions and Labels of the Tagged Particle}
	Let's recall assumptions on jump rates $p(\cdot),q(\cdot)$. 
	\begin{enumerate}[label=A''\arabic{*}] 
		\item (Supports) $p(\cdot)$ has a support on $-2,-1,1$; $q(\cdot)$ has a support on $-1,1,2$, \label{Asmp:support} 
		
		\item \textcolor{black}{(Radially decreasing)} $ p(-1) \geq p(-2)$, $q(1) \geq q(2)>0$,  \label{Asmp:attractiveness3}
		
		\item (Dominance and Positive) \label{Asmp:Dominance and Positive}	$q(1) \geq p(1), q(-1) \leq p(-1)$, $w = \sum_z z\cdot p(z) >0$.
	\end{enumerate}
	These conditions imply the tagged particle moves "faster" than a red particle, \textcolor{black}{and red particles starting from the left of the tagged particle always remains to the left of the tagged particle.
	Consider an AEP with a driven tagged particle, we label particles in an ascending order and also keep track of the label of tagged particle. We get a modified auxiliary process $(\vec{X}_t, I_t)=(\vec{X}_0,p,q, I_t)=\left(\left(X_i(t)\right)_{i\in \mathbb{Z}},I_t\right)$. Its generator $\hat{L}_{p,q}
$ is given by its action on a local function $F$, 
\begin{align} \label{eq:modified auxiliary process}
	\hat{L}_{p,q}F(\vec{X},I) =& \sum_{i\neq I,z\in\mathbb{Z}}p(z)\mathbb{1}_{A_{i,z}}(\vec{X})\left[F(T_{i,z}\vec{X},\hat{I}_{i,z}(\vec{X},I))-F(\vec{X},I)\right]  \notag \\
	&+ \sum_z q(z) \mathbb{1}_{A_{I,z}}(\vec{(X)})\left[F(T_{i,z}\vec{X},I_{I,z}(\vec{X}))-F(\vec{X},I)\right]  
\end{align}where $T_{i,z}\vec{X}$ is defined by \eqref{auxiliary exchange},\eqref{eq: reverse T_i,z}, and $\hat{I}_{i,z}(\vec{X},I)$ is defined as
\begin{equation} \hat{I}_{i,z}(\vec{X},I) =\begin{cases}
I-1, &\text{ if } X_i<X_I<X_i+z \\
I+1, &\text{ if } X_i+z<X_I<X_i \\
I_{i,z}(\vec{X}), &\text{ if } i=I, \\
I, &\text{ else. }
\end{cases} \label{eq: modified index change}
\end{equation}  
}	
	\textcolor{black}{For an AEP with a usual tagged particle, we get a second auxiliary process $(\vec{Y}_t, i_t)=(\vec{Y}_0,p,p, i_t)$ with a generator $\hat{L}_{p,p}$. For convenience, we let initial configuration be the same for both processes, and label the tagged particles with $0$, ie.
	\begin{equation}
	I_0 = i_0 =0
	\end{equation} }

	\textcolor{black}{The first lemma says that we can couple two modified auxiliary processes $(\vec{X}_t,I_t)$ and $(\vec{Y}_t,i_t)$ in the sensein the sense similar to Definition \ref{def:coupling}., for all $t\geq 0$, 
		\begin{equation} \label{eq: ordering}
		X_i(t) \geq Y_i(t), \text{ for all } i \text{ in } \mathbb{Z}.
		\end{equation} 
	\begin{lemma}\label{lemma: coupling of auxiliary process in abs frame} Suppose $p(\cdot), q(\cdot)$ satisfy  \ref{Asmp:support}, \ref{Asmp:attractiveness3} and \ref{Asmp:Dominance and Positive}. For two modified auxiliary processes $(\vec{X}_t,I_t)$ and $(\vec{Y}_t, i_t)$ with generators $\hat{L}_{p,q}$ and $\hat{L}_{p,p}$, there is a joint $\Omega$, such that if \eqref{eq: ordering} holds for $t=0$, we have \eqref{eq: ordering} holds for all $t>0$, and the marginal condition holds
		\begin{align*} \Omega F_1\left(\vec{X},I,\vec{Y},i\right) = \hat{L}_{p,q} H_1\left(\vec{X},I\right),  \\
		\Omega F_2\left(\vec{X},I,\vec{Y},i\right) = \hat{L}_{p,p} H_2\left(\vec{Y},i\right), 
		\end{align*} 
		for any local functions $F_1\left(\vec{X},I,\vec{Y},i\right)=H_1\left(\vec{X},I\right)$ and $F_2\left(\vec{X},I,\vec{Y},i\right)=H_2\left(\vec{Y},i\right)$.
	\end{lemma}
	\begin{proof}
		This is proved in Corollary \ref{cor: coupling of maprocess}. In this case, we have $R=2$.
	\end{proof}
}

	\textcolor{black}{The second lemma gives estimates of $I_t$ and $i_t$ with respect to Bernoulli initial measure $\mu_\rho$.}
	\begin{lemma}\label{lemma: change of indices} Suppose $p(\cdot), q(\cdot)$ satisfy  \ref{Asmp:support}, \ref{Asmp:attractiveness3} and \ref{Asmp:Dominance and Positive}. Let$I_0= i_0= 0$, and $\vec{X}_0$ correspond to the initial Bernoulli product measure $\mu_\rho$. The \textcolor{black}{labels} $I_t$, $i_t$ of the tagged particles in \textcolor{black}{the modified processes $(\vec{X}_t,I_t)=(\vec{X}_0,p,q,I_t)$ and $(\vec{Y}_t,i_t)=(\vec{X}_0,p,p,i_t)$} satisfy,
		\[ \liminf_{t\to \infty} \frac{I_t}{t} \geq 0, \  \mathbb{P}^{\mu_\rho,q}\--a.s. \]
		and
		\[ \lim_{t\to \infty} \frac{i_t}{t} = 0, \  \mathbb{P}^{\mu_\rho,p}\--a.s. \]
	\end{lemma} 
	\begin{proof}
		Notice that $i_t$ is identical to the integrated current $-N_t$ \textcolor{black}{through} bond $(-1,1)$ in the environment process $\xi_t$. \textcolor{black} {By $\mathbb{P}^{\mu_\rho,p}$- martingale $N_t-\int_0^t C_{-1,1}(\xi_s)\,ds$,
		\[\mathbb{E}^{\mu_\rho,p}\left[\frac{i_t}{t}\right] =\mathbb{E}^{\mu_\rho,p}\left[\frac{-N_t}{t}\right] = \mathbb{E}^{\mu_\rho,p}\left[\frac{1}{t} \int_{0}^{t} -C_{-1,1}(\xi_s)\,ds\right] = 0,\]
		where 
		\[ C_{-1,1} = -p(-2)\xi_{1}(1-\xi_{-1})+q(-2) \xi_{-1}(1-\xi_{-2}) + p(2)\xi_{-1}(1-\xi_{1}) -q(2) \xi_{1}(1-\xi_{2}).\] The last expectation is zero because $p(\cdot)=q(\cdot)$. As the Bernoulli measure $\mu_\rho$ is ergodic for $\xi_t$,
			\[ \lim_{t\to \infty} \frac{i_t}{t} =\mathbb{E}^{\mu_\rho,p}\left[\frac{i_t}{t}\right]= 0, \  \mathbb{P}^{\mu_\rho,p}\--a.s. \]}
			
		For $I_t$, the jump rates indicate ${I_t = -N_t\geq 0}$. \textcolor{black}{ See the second point in Remark \ref{rm: general jumprates}.}
	\end{proof}
	
	\subsection{Proof of Theorem \ref{Thm:ballistic behavior for fast moving tagged particle}}
	Now we can prove Theorem \ref{Thm:ballistic behavior for fast moving tagged particle}.
	
	\begin{proof}(Theorem \ref{Thm:ballistic behavior for fast moving tagged particle})
		By Lemma \ref{lemma: coupling of auxiliary process in abs frame}, there is a joint distribution $\mathbb{P}$, and we have $\vec{X}_t \geq \vec{Y}_t,\ \mathbb{P}\--a.s.$ Particularly, $X_{I_t} \geq Y_{I_t},\ \mathbb{P}\--a.s.$
		
		On the other hand, by Lemma \ref{lemma: change of indices}, under the joint distribution $\mathbb{P}$, \textcolor{black}{ which has marginal distributions $\mathbb{P}^{\mu_\rho,q}$ and $\mathbb{P}^{\mu_\rho,p}$,}
		\[ \liminf_{t\to \infty} \frac{I_t-i_t}{t} \geq 0, \  \mathbb{P}\--a.s.  \] 
		
		Fix any $\delta >0$, $I_t \geq \lfloor i_t- \delta \cdot t \rfloor $ for large $t$, so $Y_{I_t} \geq Y_{\lfloor i_t-\delta\cdot t\rfloor}$. Consider $Y_{i_t}-Y_{\lfloor i_t-\delta\cdot t\rfloor}$. \textcolor{black}{Since the Bernoulli product measure $\mu_\rho$ is an ergodic measure for the environment process,} $Y_{i_t}-Y_{\lfloor i_t-\delta\cdot t\rfloor}$ is dominated by the sum of $\lceil \delta \cdot t \rceil$ independent geometric random variables with parameter $\rho$. \textcolor{black}{We therefore get}
		\[ \limsup_{n\to \infty} \frac{Y_{i_{t_n}}-Y_{\lfloor i_{t_n}-\delta\cdot t_n\rfloor}}{t_n} \leq \frac{ \delta}{\rho}, \  \mathbb{P}\--a.s.\] for some sequence $t_n = \frac{n}{2^{k}} \uparrow \infty$. A standard interpolation argument enables us to replace "$t_n \uparrow \infty$" by "$t\uparrow \infty$".
		
		 With \textcolor{black}{the law of large numbers for a tagged particle in usual AEP, $\lim_{t\to\infty}\frac{Y_{i_{t}}}{t} = w\cdot (1-\rho) $},  we have
		\[ \liminf_{t\to \infty} \frac{Y_{I_t}}{t} \geq \liminf_{t\to \infty} \frac{Y_{\lfloor i_t-\delta\cdot t\rfloor}}{t}\geq w\cdot (1-\rho) - \frac{\delta}{\rho}, \  \mathbb{P}\--a.s.\] 
		where $\textcolor{black}{w = \sum_z z\cdot p(z)}  > 0$. This is sufficient to get Theorem \ref{Thm:ballistic behavior for fast moving tagged particle} since $X_{I_t} \geq Y_{I_t}$.
	\end{proof}

\appendix

\section{APPENDIX} \label{sec: existence of coupling}
	The \textcolor{black}{generator for the coupled process in Theorem \ref{thm:coupling} is long} and consists of several parts. We shall first prove three lemmas. Particularly, Lemma \ref{lm: coupling preserve i-th} is the main step towards the construction of the coupling.
	
	Firstly, we observe that these are jump processes. Because the generators are additive for the same type of $G =\tilde{L}, \tilde{L}_L,$ or $\tilde{L}_R$, we can combine two pairs of coupled processes \textcolor{black}{, in the sense of adding their generators,} to obtain a new pair. \textcolor{black}{The main requirement is that couplings exist for any ordered deterministic initial configurations.}
	
	\begin{lemma} \label{lm:combining two pairs}
		Suppose we have two joint generators $\Omega_1$, $\Omega_2$, for two pairs of auxiliary processes. \textcolor{black}{If under $\Omega_1$, $\Omega_2$, the two pairs of} auxiliary processes satisfy Definition \ref{def:coupling} \textcolor{black}{for} any (deterministic) $\vec{W}_0 \geq \vec{X_0}$, that is,
		\[ \vec{W}_t \succeq\vec{X}_t, \vec{Y}_t \succeq \vec{Z}_t, \]
		where 
		\[\vec{W}_t = (\vec{W}_0,G,p_1,q_1),\vec{X}_t = (\vec{X}_0,G',p_2,q_2)\]
		and
		\[\vec{Y}_t = (\vec{W}_0,G,p'_1,q'_1),\vec{Z}_t = (\vec{X}_0,G',p'_2,q'_2).\] 
		Then, the combined auxiliary processes $\vec{U}_t, \vec{V}_t$, are also coupled via the joint generator $\Omega = \Omega_1 + \Omega_2$ that is:
		\[\vec{U}_t \succeq\vec{V}_t,\]
		where \textcolor{black}{
		\[\vec{U}_t = (\vec{W}_0,G,p_1+p'_1,q_1+q'_1),\vec{V}_t = (\vec{X}_0,G',p_2+p'_2,q_2+q'_2).\]} 
		\textcolor{black}{We can use either $p(\cdot)$ or $p(\cdot,\cdot)$ in this context, and generators $G$, $G'$ can be the same.}
	\end{lemma}
	\begin{proof}	By assumption, the condition for the marginals is immediate from the forms of the generators \eqref{Eq:auxiliary generator} \eqref{Eq:auxiliary generator-Right} and \eqref{Eq:auxiliary generator-Left}. We need to check the first condition.
		
		By arguments in the proof of Theorem 2.5.2 \cite{KL}, to show $\vec{W}_t \geq \vec{X}_t$, we need to show the closed set $F_0=\{(\vec{X},\vec{Y}):\vec{X} \geq \vec{Y}\}$ is an absorbing set, which can be checked via showing:
		\begin{equation} \label{Cond: obsorbing}
		\Omega\mathbb{1}_{F_0} \geq 0.
		\end{equation}
		\textcolor{black}{Indeed, by martingale $\mathbb{1}_{F_0}(\vec{X}_t,\vec{Y}_t) - \int_0^t \Omega\mathbb{1}_{F_0}(\vec{X}_s,\vec{Y}_s)\,ds$, we get from \eqref{Cond: obsorbing}, for any $t \geq 0$
			\[P(\vec{X}_t\geq \vec{Y}_t) = \mathbb{E}\left[\mathbb{1}_{F_0}(\vec{X}_t,\vec{Y}_t)\right] \geq  \mathbb{E}\left[\mathbb{1}_{F_0}(\vec{X}_0,\vec{Y}_0)\right] =1. \] Usual interpolation arguments allow us to get $P(\vec{X}_t\geq \vec{Y}_t, \text{ for all } t)=1.$}
		
		\textcolor{black}{
		Lastly, by direct computation, we get 
		\[\Omega_1\mathbb{1}_{F_0} \geq 0, \text{ and } \Omega_2\mathbb{1}_{F_0} \geq 0,\] which is sufficient for \eqref{Cond: obsorbing}.}
	\end{proof}
	
	Secondly, we observe four monotone functions on the configuration space by comparing the configurations before and after the tagged particle jump \textcolor{black}{with shifts of labels or not}. See Figure 2,3 for examples.
	\begin{lemma}\label{lm:jumps preserve ordering} Let z>0. If a jump of the tagged particles by $z$ or $-z$ is possible, we have
		\begin{align}
		\Theta_{-z}\vec{X} \geq \vec{X} ,&\quad	S_z\circ\Theta_z\vec{X} \geq \vec{X} \label{eq:increasing with tagged jumps}	 \\
		\Theta_{z}\vec{X} \leq \vec{X}	,& \quad S_{-z}\circ\Theta_{-z}\vec{X} \leq \vec{X}
		\end{align}
	As a consequence, there are two generators  $\Omega_{0,R}$ and $\Omega_{0,L}$, such that for any  $\vec{X}_0 \geq \vec{Y}_0$, we can couple $\vec{X_t}=(\vec{X}_0,\tilde{L}_R,0,q) \succeq \vec{Y_t}=(\vec{Y}_0,\tilde{L}_R,0,0)$  via $\Omega_{0,R}$, and couple $\vec{W_t}=(\vec{X}_0,\tilde{L},0,0) \succeq \vec{Z_t}=(\vec{Y}_0,\tilde{L}_R,0,q)$  via $\Omega_{0,L}.$ 
	\end{lemma}
	\begin{proof} \textcolor{black}{We will prove equations \eqref{eq:increasing with tagged jumps} and define $\Omega_{0,R}$, via which we can couple two auxiliary processes $\vec{X_t}=(\vec{X}_0,\tilde{L}_R,0,q) \succeq \vec{Y_t}=(\vec{Y}_0,\tilde{L}_R,0,0)$ for any initial $\vec{X}_0 \geq \vec{Y}_0$. The other case is similar.}
		
		By \eqref{auxiliary translation} and \eqref{shift labels}, \textcolor{black}{we check coordinates},  	
		\begin{align}
		(\Theta_{-z}\vec{X})_i = X_i + z \geq X_i \notag\\ 
		(S_z\circ\Theta_z\vec{X})_i = X_{i+z}-z\geq X_i		
		\end{align}
		Then it is easy to see the generator $\Omega_{0,R}$  \textcolor{black}{ defined below works}, since under this generator, \textcolor{black}{$\vec{X}_t$ is increasing in $t$ while $\vec{Y}_t$ is constant in $t$},
		\begin{align}
		\Omega_{0,R}F(\vec{X},\vec{Y}) =& \tilde{L}_R F(\cdot,\vec{Y})\left[\vec{X}\right] \notag \\
		=  &\sum_{y<0}q(y)\mathbb{1}_{B_{y}}(\vec{X})\left[F(\Theta_y\vec{X},\vec{Y})-F(\vec{X},\vec{Y})\right] \notag \\
		+&\sum_{y>0}q(y)\mathbb{1}_{B_{y}}(\vec{X})\left[F(S_y\circ\Theta_y\vec{X},\vec{Y})-F(\vec{X},\vec{Y})\right].
		\end{align}
		
	\end{proof}
	
	Thirdly, we see that given $\vec{X} \geq \vec{Y}$, whenever the i-th particle in $\vec{Y}$ jump by $z>0$, we can move the i-th particle in $\vec{X}$ by $z' \geq 0$, \textcolor{black}{such that ordering is preserved after relabeling, $T_{i,z'}\vec{X}\geq T_{i,z}\vec{Y}$}.  \textcolor{black}{This is the primary step for constructing couplings in Theorem \ref{Thm: main step}. Once we can couple positive jumps of the i-th particle in the slower process by positive jumps of its corresponding particle in the faster process, we only need to assign jump rates according to different pairs $z,z'$. The assignment is possible by Assumptions \ref{Asmp:attractiveness2}, \ref{Asmp:finite-range2}.} See \eqref{def:p_i,s,z}, \eqref{def:p_i,s,0} in Theorem \eqref{Thm: main step} for assignment in detail.
	\begin{lemma}\label{lm: coupling preserve i-th}
		Assume $\vec{X} \geq \vec{Y}$, $R \geq z>0$, and \textcolor{black}{$\vec{Y}\in A_{i,z}$}, then there is a $z' \geq 0$ depending on $\vec{X},\vec{Y}$,$i$,and $z$, such that \textcolor{black}{$\max\{Y_i+z,X_i\}\geq X_i+z'\geq \min\{Y_i+z,X_i\}$} and
		\begin{equation}
		\vec{X'}= T_{i,z'}\vec{X} \geq T_{i,z}\vec{Y}=\vec{Y'}
		\end{equation}
		\textcolor{black}{Particularly, any $z'\neq 0$ is obtained from at most one $z$.}
	\end{lemma}
	\begin{proof}
		We first describe how to find $z'$, and then prove by induction.
		\begin{enumerate}[label=Step \arabic*]
			\item	\textcolor{black}{Suppose there are exactly $k$ holes in $\vec{Y}$ between $Y_i$ and $Y_i+R$. We label them} in a descending order:
			\begin{equation}
			Y_i < \textcolor{black}{Y_i+z_k < Y_i+z_{k-1}} <\dots < Y_i+z_1 \leq Y_i+R
			\end{equation} 
			Define \textcolor{black}{$z'_l, l= 1,2,\dots, \textcolor{black}{k}$} inductively by,
			\begin{align} 
			z'_1 =&\begin{cases} \max\{z'> 0: X_i + z' \leq Y_i +z_1,\vec{X}\in \textcolor{black}{A_{i,z'}}\} &, \text{if exists}\\
			0 &,\text{else}
			\end{cases} \\
			z'_{l+1} =& \begin{cases} \max\{z'_l > z'> 0: X_i + z' \leq Y_i + \textcolor{black}{z_l},\vec{X}\in \textcolor{black}{A_{i,z'}}\} &, \text{if exists}\\
			0 &,\text{else}
			\end{cases}
			\end{align} That is, if $z_l'>0$, $X_i+ z'_l$ is the right-most hole in $\vec{X}$ which is to the left of the (l-1)-th hole in $\vec{X}$ \textcolor{black}{selected so far} and the l-th hole in $\vec{Y}$. See Figure 6 for an example. In this example, $i=0$, $R =8$, $z'_1=7$, $z'_2=3$, $z'_3=1$, $z'_4=0$.
			
			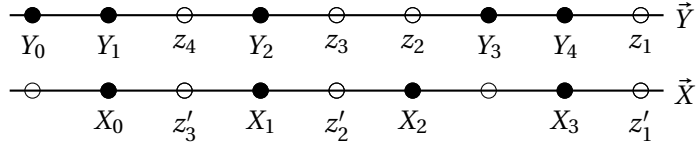
\begin{figure}[h!]
				\begin{center}
					\begin{tikzpicture} 
					\draw[black, thick] (-4.3,0) -- (4.3,0)node[black,right] {$\vec{Y}$};  
					
					\foreach \x in {-4,...,4}
					\draw (-\x,0) circle (0.1);
					\draw [black,fill] (-4,0) circle [radius=0.1] node [black,below=4] {$Y_{0}$}; 
					\draw [black,fill] (-3,0) circle [radius=0.1] node [black,below=4] {$Y_{1}$}; 
					\draw [black,fill] (-1,0) circle [radius=0.1] node [black,below=4] {$Y_{2}$}; 
					\draw [black,fill] (2,0) circle [radius=0.1] node [black,below=4] {$Y_{3}$}; 
					\draw [black,fill] (3,0) circle [radius=0.1] node [black,below=4] {$Y_{4}$}; 
					\draw [black,fill] (4,0) circle; 
					
					\draw (4,0) circle (0.1)node [black,below=4] {$z_{1}$};
					\draw (1,0) circle (0.1)node [black,below=4] {$z_{2}$};
					\draw (0,0) circle (0.1)node [black,below=4] {$z_{3}$};
					\draw (-2,0) circle (0.1)node [black,below=4] {$z_{4}$};

					\draw[black, thick] (-4.3,-1) -- (4.3,-1) node[black,right] {$\vec{X}$};  
					
					\foreach \x in {-4,...,4}
					\draw (-\x,0-1) circle (0.1);
					\draw [black,fill] (-3,0-1) circle [radius=0.1] node [black,below=4] {$X_{0}$}; 
					\draw [black,fill] (-1,0-1) circle [radius=0.1] node [black,below=4] {$X_{1}$}; 
					\draw [black,fill] (1,0-1) circle [radius=0.1] node [black,below=4] {$X_{2}$}; 
					\draw [black,fill] (3,0-1) circle [radius=0.1] node [black,below=4] {$X_{3}$}; 
					\draw (4,0-1) circle (0.1)node [black,below=4] {$z'_{1}$};
					\draw (0,0-1) circle (0.1)node [black,below=4] {$z'_{2}$};
					\draw (-2,0-1) circle (0.1)node [black,below=4] {$z'_{3}$};
					\end{tikzpicture}
				\end{center}
				\caption{Target Sites $z'$ for $X_0$}
			\end{figure}
			
			\item We then show the base case for $l=1$. That is, we want to show $\vec{X'}= T_{i,z_1'}\vec{X} \geq T_{i,z_1}\vec{Y}=\vec{Y'}$ (for any $k>0$).
			\begin{enumerate}
				\item \textcolor{black}{$X_i > Y_i+z_1$} and $z_1' =0$: \textcolor{black}{By \eqref{eq:index of right jump}, \eqref{auxiliary exchange}}, we have, 
				\begin{align*}
				Y'_j = Y_j \leq X_j = X'_j,  &\quad j < i \text{ or } j >I_1,\\
				\textcolor{black}{Y'_j \leq Y'_{I_1} = Y_i+z_1 <X_i \leq X_j = X'_j},& \quad   i\leq j \leq I_1,
				\end{align*}
				\textcolor{black}{where $I_1 = I_{i,z_1}(\vec{Y})= \max\{s: Y_s \leq Y_i+z_1\}$.}
				\item \textcolor{black}{$X_i \leq Y_i+z_1$} and $z_1' \geq 0$: Let $I_1 = I_{i,z_1}(\vec{Y})$, $I_2 = I_{i,z'_1}(\vec{X})$, and $r = \max\{s\geq I_2: X_s \leq \textcolor{black}{Y_i}+ z_1  \}$. We have, $I_2 \leq r\leq  I_1$. Indeed, every particle in $\vec{X}$ between sites $X_i$ and $Y_i+z_1$ corresponds to \textcolor{black}{ a particle in $\vec{Y}$ between sites $Y_i$ and $Y_i+z_1$,
				\[\left\{s: X_i\leq X_s\leq Y_i+z_1\right\} \subset \left\{s: Y_i\leq Y_s\leq Y_i+z_1\right\},\]} and we obtain $r-i+1 \leq I_1-i+1$.  
				
				If $ I_2<r$, \textcolor{black}{ by definition of $z_1'$, we get} $X_r = Y_i+z_1 =Y'_{I_1} $, and 
				\begin{align*}
				Y'_j = Y_j \leq X_j = X'_j,  &\quad j < i \text{ or } j >I_1,\\
				Y'_j = Y_{j+1} \leq X_{j+1} = X'_j,& \quad   i\leq j \leq I_2-1, \\
				Y'_j \leq Y'_{I_1} - (I_1-j) \textcolor{black}{\leq} X_{r}-(r-j) \textcolor{black}{= X'_j},& \quad   I_2\leq j \leq r,\\
				\textcolor{black}{Y'_j \leq Y'_{I_1} = Y_i+z_1 =X_r \leq X'_j},& \quad   r< j \leq I_1.
				\end{align*}
				\textcolor{black}{The last equality in the third case is because there is exactly one hole in $\vec{X}$ (at site $X_i+z_1'$) between sites $X_i+z_1'$ and $X_r$.}
				
				\textcolor{black}{Otherwise, if $I_2 = r$, by definitions of $z_1',r,$ and $I_2$, we have $X'_r= Y_i+z_1 = X_i+z_1' = X'_{I_2}$, and $I_1 =I_2$}. Therefore,
				\begin{align*}
				Y'_j = Y_j \leq X_j = X'_j,  &\quad j < i \text{ or } j >I_1,\\
				Y'_j = Y_{j+1} \leq X_{j+1} = X'_j,& \quad   i\leq j \leq I_1-1, \\
				Y'_{I_1}=  X'_{I_1},& \quad   j =I_1=I_2.
				\end{align*}
				
				One may work with an example, see Figure 7. In this case, $ R =8, i =0, \textcolor{black}{I_1 = 5},I_2=2,r=4$.
				
				\begin{figure}[h!]
					\begin{center}
						\begin{tikzpicture} 
						\draw[black, thick] (-4.3,0) -- (4.3,0)node[black,right] {$\vec{Y}$};  
						
						\foreach \x in {-4,...,4}
						\draw (-\x,0) circle (0.1);
						\draw [black,fill] (-4,0) circle [radius=0.1] node [black,below=4] {$Y_{0}$}; 
						\draw [black,fill] (-3,0) circle [radius=0.1] node [black,below=4] {$Y_{1}$};
						
						\draw [black,fill] (-1,0) circle [radius=0.1] node [black,below=4] {$Y_{2}$}; 
						\draw [black,fill] (0,0) circle [radius=0.1] node [black,below=4] {$Y_{3}$}; 
						\draw [black,fill] (1,0) circle [radius=0.1] node [black,below=4] {$Y_{4}$};
						 \draw [black,fill] (2,0) circle [radius=0.1] node [black,below=4] {$Y_{5}$};
						\draw [black,fill] (4,0) circle [radius=0.1] node [black,below=4] {$Y_{6}$};

						\draw (-2,0) circle (0.1)node [black,below=4] {$z_{2}$};
						\draw (3,0) circle (0.1)node [black,below=4] {$z_{1}$};
						
						\draw[dashed,->] (-4,0.3) to (-4,0.8) to (3,0.8)
						to (3,0.3);

						\draw[black, thick] (-4.3,-1) -- (4.3,-1) node[black,right] {$\vec{X}$};  
						
						\foreach \x in {-4,...,4}
						\draw (-\x,0-1) circle (0.1);
						\draw [black,fill] (-3,0-1) circle [radius=0.1] node [black,below=4] {$X_{0}$}; 
						\draw [black,fill] (-2,0-1) circle [radius=0.1] node [black,below=4] {$X_{1}$}; 
						\draw [black,fill] (0,0-1) circle [radius=0.1] node [black,below=4] {$X_{2}$}; 
						\draw [black,fill] (2,0-1) circle [radius=0.1] node [black,below=4] {$X_{3}$}; 
						\draw [black,fill] (3,0-1) circle [radius=0.1] node [black,below=4] {$X_{4}$}; 
						\draw [black,fill] (4,0-1) circle;
						\draw (1,0-1) circle (0.1)node [black,below=4] {$z'_{1}$};
						\draw[dashed,->] (-3,-1.3-0.3) to (-3,-1.3-0.8) to (1,-1.3-0.8)
						to (1,-1.3-0.3);
						
						\draw[black, thick] (-4.3,-3) -- (4.3,-3)node[black,right] {$T_{0,z_1}\vec{Y}$};  
						
						\foreach \x in {-4,...,4}
						\draw (-\x,-3) circle (0.1);
						\draw [black,fill] (-3,-3) circle [radius=0.1] node [black,below=4] {$Y_{0}$}; 
						\draw [black,fill] (-1,-3) circle [radius=0.1] node [black,below=4] {$Y_{1}$}; 
						\draw [black,fill] (0,-3) circle [radius=0.1] node [black,below=4] {$Y_{2}$}; 
						\draw [black,fill] (1,-3) circle [radius=0.1] node [black,below=4] {$Y_{3}$}; 
						\draw [black,fill] (2,-3) circle [radius=0.1] node [black,below=4] {$Y_{4}$}; 
						\draw [black,fill] (3,-3) circle [radius=0.1] node [black,below=4] {$Y_{5}$};
						 \draw [black,fill] (4,-3) circle [radius=0.1] node [black,below=4] {$Y_{6}$};  
						
						\draw[black, thick] (-4.3,-4) -- (4.3,-4)node[black,right] {$T_{0,z'_1}\vec{X}$};  
						\foreach \x in {-4,...,4}
						\draw (-\x,0-4) circle (0.1);
						\draw [black,fill] (-2,0-4) circle [radius=0.1] node [black,below=4] {$X_{0}$}; 
						\draw [black,fill] (0,0-4) circle [radius=0.1] node [black,below=4] {$X_{1}$}; 
						\draw [black,fill] (1,0-4) circle [radius=0.1] node [black,below=4] {$X_{2}$}; 
						\draw [black,fill] (2,0-4) circle [radius=0.1] node [black,below=4] {$X_{3}$}; 
						\draw [black,fill] (3,0-4) circle [radius=0.1] node [black,below=4] {$X_{4}$}; 
						\end{tikzpicture}
					\end{center}
					\caption{Configurations before and after Jumps $z_1,z'_1$}
				\end{figure}
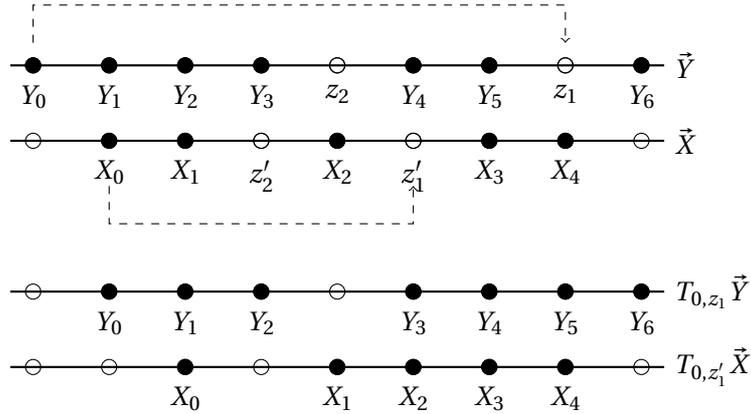
				
			\end{enumerate}

			\item For the inductive step, we assume \textcolor{black}{that $n \leq k-1$, we have $T_{i,z'_l}\vec{X} \geq   T_{i,z_l}\vec{Y}$ for $l =1,2,\dots, n$, and} we want to show  $\vec{X}'= T_{i,z'_{n+1}}\vec{X} \geq   T_{i,z_{n+1}}\vec{Y}=\vec{Y}'.$
			
			Let $I_1' = I_{i,\textcolor{black}{z_{n+1}}}(\vec{Y})$, and $I_2' =I_{i,\textcolor{black}{z'_{n+1}}}(\vec{X})$.  $I_2' \leq I_1'$ \textcolor{black}{since we can get $X_{I'_2}< X_{I'_1+1}$. Indeed, by definition of $z'_{n+1}, r, I'_1$, and $I'_2$, we have.\begin{equation}\label{eq:easy checking}
				{X_{I'_2}\leq X'_{I'_2}=X_i+ z'_{n+1} \leq Y_i+z_{n+1} = Y'_{I'_1}<Y'_{I'_1+1}=Y_{I'_1+1}\leq X_{I'_1+1}}\end{equation}} It suffices to show 
			\begin{equation} \label{eq: target} Y'_j = \left(T_{i,z_{n+1}}\vec{Y})\right)_j \leq \left(T_{i,z'_{n+1}}\vec{X})\right)_j =X'_j
			\end{equation}
			for $ I'_2 \leq j \leq I_1'$. For other $j$, we use the same inequalities as the first two in Step 2 (b) and get $Y_j' \leq X'_j$ \textcolor{black}{when $j<i$, $j>I'_1$, or $i\leq j < I'_2$.}
			
			\textcolor{black}{To get \eqref{eq: target}, we can consider a new configuration $K(\vec{X},s)$ by adding a particle to an empty site s, and relabeling all the particles to the left of the site s}. See Figure 8 for an example. 
			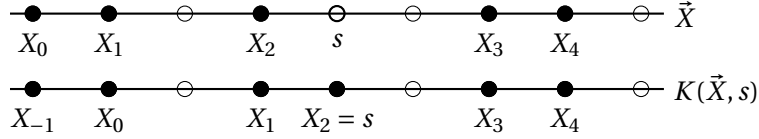
\begin{figure}[h!]
				\begin{center}
					\begin{tikzpicture} 
					\draw[black, thick] (-4.3,0) -- (4.3,0)node[black,right] {$\vec{X}$};  
					
					\foreach \x in {-4,...,4}
					\draw (-\x,0) circle (0.1);
					\draw [black,fill] (-4,0) circle [radius=0.1] node [black,below=4] {$X_{0}$}; 
					\draw [black,fill] (-3,0) circle [radius=0.1] node [black,below=4] {$X_{1}$}; 
					\draw [black,fill] (-1,0) circle [radius=0.1] node [black,below=4] {$X_{2}$}; 
					\draw [black,fill] (2,0) circle [radius=0.1] node [black,below=4] {$X_{3}$}; 
					\draw [black,fill] (3,0) circle [radius=0.1] node [black,below=4] {$X_{4}$}; 
					\draw [black,thick] (0,0) circle [radius=0.1] node [black,below=4] {$s$};
					
					\draw[black, thick] (-4.3,-1) -- (4.3,-1) node[black,right] {$K(\vec{X},s)$};  
					
					\foreach \x in {-4,...,4}
					\draw (-\x,0-1) circle (0.1);
					\draw [black,fill] (-4,0-1) circle [radius=0.1] node [black,below=4] {$X_{-1}$}; 
					\draw [black,fill] (-3,0-1) circle [radius=0.1] node [black,below=4] {$X_{0}$}; 
					\draw [black,fill] (-1,0-1) circle [radius=0.1] node [black,below=4] {$X_{1}$};
					\draw [black,fill] (2,0-1) circle [radius=0.1] node [black,below=4] {$X_{3}$}; 
					\draw [black,fill] (3,0-1) circle [radius=0.1] node [black,below=4] {$X_{4}$}; 
					\draw [black,fill] (0,0-1) circle [radius=0.1] node [black,below=4] {$X_{2} =s$}; 
					\end{tikzpicture}
				\end{center}
				\caption{Adding a Particle to a Vacant Site}
			\end{figure}
			
			Let $m = \max\{k:X_k \leq s\}$, and $s \not\in \vec{X}$,
			\begin{equation} \label{eq:K transform}
			(K(\vec{X},s))_j\textcolor{black}{:=} \begin{cases}
			X_j, &\text{ if } j>m,\\
			s, &\text{ if } j=m,\\
			X_{j+1}, &\text{ if } j<m
			\end{cases} 
			\end{equation}
			if $s\in \vec{X}$, we use convention $K(\vec{X},s) =\vec{X}$. 
			
			\begin{enumerate}
				\item \textcolor{black}{If $z'_1>0$, the  site $X_i$ in $T_{i,z'_1}\vec{X}$ is empty, and so is $Y_i$ in $T_{i,z_1}\vec{Y}$}. By \textcolor{black}{definition} \eqref{eq:K transform}, we have
				\begin{align*} K(T_{i,z'_1}\vec{X},X_i) \geq K(T_{i,\textcolor{black}{z_1}}\vec{Y},Y_i).\end{align*}
				 \textcolor{black}{$Y_i+z_{n+1}$ is the n-th hole from the right in the new configuration $K(T_{i,\textcolor{black}{z_1}}\vec{Y},Y_i)$ between site $Y_i$ and site $Y_i+R$, and the additional particle is the (i-1)-th at site $Y_i$.}  It is similar for \textcolor{black}{$K(T_{i,z'_{n+1}}\vec{X},X_i)$}.
				By the inductive hypothesis, we have
				\begin{equation} \label{eq:inductive hyp cons}
				T_{\textcolor{black}{i-1,z'_{n+1}}}K(T_{i,z'_1}\vec{X},X_i) \geq T_{\textcolor{black}{i-1,z_{n+1}}}K(T_{i,z_1}\vec{Y},Y_i). \end{equation}
				
				Let $I_1 = I_{i,z_1}\vec{Y}, I_2 = I_{i,z'_1}\vec{X}$. By an argument similar to \eqref{eq:easy checking}, $I_1\geq I_2$, $I_1\geq I'_1$, and $I_2\geq I'_2$. \textcolor{black}{For $I_2'\leq j\leq I_1'$, consider the j-th particle in the configurations $T_{i,z'_{n+1}}\vec{X}$ and $T_{i,z_{n+1}}\vec{Y}$,} 
				\begin{align*}
				Y'_j =  \left(T_{\textcolor{black}{i-1,z_{n+1}}}K(T_{i,z_1}\vec{Y},Y_i)\right)_{\textcolor{black}{j-1}},& \quad \text{if }I_2' \leq j \leq I_1',
				\\
				X'_j = \left(T_{\textcolor{black}{i-1,z'_{n+1}}}K(T_{i,z'_1}\vec{X},X_i)\right)_{\textcolor{black}{j-1}},&\quad\text{if } I_2' \leq j \textcolor{black}{\leq} I_2,\\
				X'_j \textcolor{black}{=} \left(T_{i,z'_{k+1}}K(T_{i,z'_1}\vec{X},X_i)\right)_j,&\quad \text{if } I_2 \textcolor{black}{< j \leq I_1'}.
				\end{align*} The last inequality may be vacuously true since $I_1' \leq I_2$ is possible. By the above three inequalities and \eqref{eq:inductive hyp cons}, we get \eqref{eq: target}. See Figure 9 for an example. In this case, $R=8$, $k=3$, $n=1$, $I_1=4$, $I_2=2$,$I'_1=3$, $I'_2=2$. 
				\begin{figure}[h!]
					\begin{center}
						\begin{tikzpicture} 
						\draw[black, thick] (-4.3,0) -- (4.3,0)node[black,right] {$\vec{Y}$};  
						
						\foreach \x in {-4,...,4}
						\draw (-\x,0) circle (0.1);
						\draw [black,fill] (-4,0) circle [radius=0.1] node [black,below=4] {$Y_{0}$}; 
						\draw [black,fill] (-3,0) circle [radius=0.1] node [black,below=4] {$Y_{1}$};
						\draw [black,fill] (-2,0) circle [radius=0.1] node [black,below=4] {$Y_{2}$};  
						\draw [black,fill] (0,0) circle [radius=0.1] node [black,below=4] {$Y_{3}$}; 
						\draw [black,fill] (2,0) circle [radius=0.1] node [black,below=4] {$Y_{4}$};
						\draw [black,fill] (4,0) circle [radius=0.1] node [black,below=4] {$Y_{5}$}; 
						
						\draw (3,0) circle (0.1)node [black,below=4] {$z_{1}$};
						\draw (1,0) circle (0.1)node [black,below=4] {$z_{2}$};
						\draw (-1,0) circle (0.1)node [black,below=4] {$z_{3}$};
						
						\draw[dashed,->] (-4,0.3) to (-4,0.8) to (1,0.8)
						to (1,0.3);

						\draw[black, thick] (-4.3,-1) -- (4.3,-1) node[black,right] {$\vec{X}$};  
						
						\foreach \x in {-4,...,4}
						\draw (-\x,0-1) circle (0.1);
						\draw [black,fill] (-3,0-1) circle [radius=0.1] node [black,below=4] {$X_{0}$}; 
						\draw [black,fill] (-2,0-1) circle [radius=0.1] node [black,below=4] {$X_{1}$}; 
						\draw [black,fill] (-1,0-1) circle [radius=0.1] node [black,below=4] {$X_{2}$}; 
						\draw [black,fill] (2,0-1) circle [radius=0.1] node [black,below=4] {$X_{3}$}; 
						\draw [black,fill] (3,0-1) circle [radius=0.1] node [black,below=4] {$X_{4}$}; 
						\draw [black,fill] (4,0-1) circle;
						\draw (1,0-1) circle (0.1)node [black,below=4] {$z'_{1}$};
						\draw (0,0-1) circle (0.1)node [black,below=4] {$z'_{2}$};
						\draw[dashed,->] (-3,-1.3-0.3) to (-3,-1.3-0.8) to (0,-1.3-0.8)
						to (0,-1.3-0.3);
						
						\draw[black, thick] (-4.3,-3) -- (4.3,-3)node[black,right] {$T_{0,z_2}\vec{Y}$};  
						
						\foreach \x in {-4,...,4}
						\draw (-\x,-3) circle (0.1);
						\draw [black,fill] (-3,-3) circle [radius=0.1] node [black,below=4] {$Y_{0}$}; 
						\draw [black,fill] (-2,-3) circle [radius=0.1] node [black,below=4] {$Y_{1}$}; 
						
						\draw [black,fill] (0,-3) circle [radius=0.1] node [black,below=4] {$Y_{2}$}; 
						\draw [black,fill] (1,-3) circle [radius=0.1] node [black,below=4] {$Y_{3}$}; 
						\draw [black,fill] (2,-3) circle [radius=0.1] node [black,below=4] {$Y_{4}$}; 
						\draw [black,fill] (4,-3) circle [radius=0.1] node [black,below=4] {$Y_{5}$};  
						
						\draw[black, thick] (-4.3,-4) -- (4.3,-4)node[black,right] {$T_{0,z'_2}\vec{X}$};  
						\foreach \x in {-4,...,4}
						\draw (-\x,0-4) circle (0.1);
						\draw [black,fill] (-2,0-4) circle [radius=0.1] node [black,below=4] {$X_{0}$}; 
						\draw [black,fill] (-1,0-4) circle [radius=0.1] node [black,below=4] {$X_{1}$}; 
						\draw [black,fill] (0,0-4) circle [radius=0.1] node [black,below=4] {$X_{2}$}; 
						\draw [black,fill] (2,0-4) circle [radius=0.1] node [black,below=4] {$X_{3}$}; 
						\draw [black,fill] (3,0-4) circle [radius=0.1] node [black,below=4] {$X_{4}$}; 
						
						\draw[black, thick] (-4.3,-5) -- (4.3,-5)node[black,right] {$T_{0,z_2}K(T_{0,z_1}\vec{Y},Y_0)$};  
						
						\foreach \x in {-4,...,4}
						\draw (-\x,-5) circle (0.1);
						\draw (-4,-5) circle [radius=0.1] node [black,below=4] {$s$};
						
						\draw [black,fill] (-3,-5) circle [radius=0.1] node [black,below=4] {$Y_{-1}$}; 
						\draw [black,fill] (-2,-5) circle [radius=0.1] node [black,below=4] {$Y_{0}$}; 
						
						\draw [black,fill] (0,-5) circle [radius=0.1] node [black,below=4] {$Y_{1}$}; 
						\draw [black,fill] (1,-5) circle [radius=0.1] node [black,below=4] {$Y_{2}$}; 
						\draw [black,fill] (2,-5) circle [radius=0.1] node [black,below=4] {$Y_{3}$}; 
						\draw [black,fill] (3,-5) circle [radius=0.1] node [black,below=4] {$Y_{4}$};  
						\draw [black,fill] (4,-5) circle [radius=0.1] node [black,below=4] {$Y_{5}$}; 
						
						\draw[black, thick] (-4.3,-6) -- (4.3,-6)node[black,right] {$T_{0,z'_2}K(T_{0,z'_1}\vec{X},X_0)$};  
						\foreach \x in {-4,...,4}
						\draw (-\x,0-6) circle (0.1);
						\draw (-3,0-6) circle [radius=0.1] node [black,below=4] {$s$};
						
						\draw [black,fill] (-2,0-6) circle [radius=0.1] node [black,below=4] {$X_{-1}$}; 
						\draw [black,fill] (-1,0-6) circle [radius=0.1] node [black,below=4] {$X_{0}$}; 
						\draw [black,fill] (0,0-6) circle [radius=0.1] node [black,below=4] {$X_{1}$}; 
						\draw [black,fill] (1,0-6) circle [radius=0.1] node [black,below=4] {$X_{2}$}; 
						\draw [black,fill] (2,0-6) circle [radius=0.1] node [black,below=4] {$X_{3}$};
						\draw [black,fill] (3,0-6) circle [radius=0.1] node [black,below=4] {$X_{4}$};	
						\end{tikzpicture}
					\end{center}
					\caption{Configurations before and after Jumps $z_2,z'_2$}
				\end{figure}
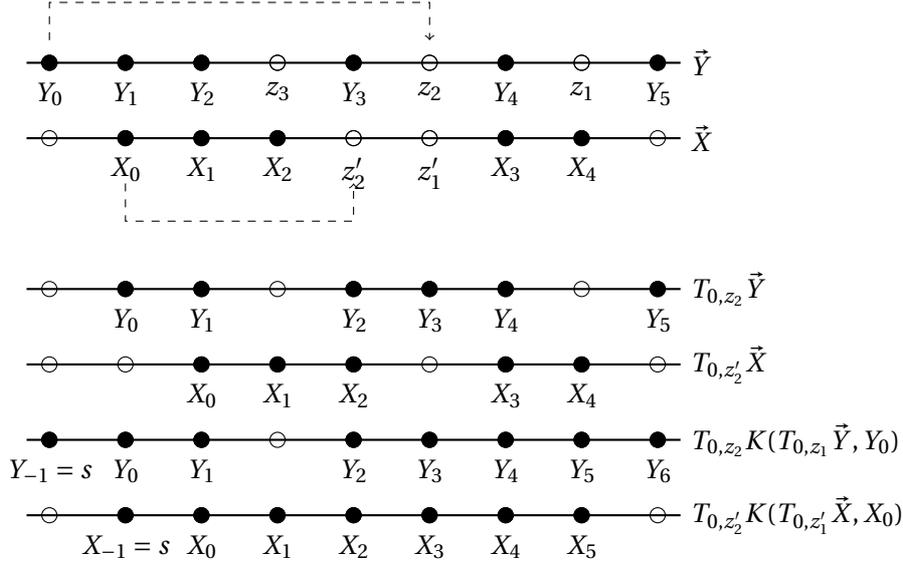
				\item if $z'_1=0$, it is easy to see $T_{i,z_1}\vec{Y}\geq T_{i,z_{\textcolor{black}{n+1}}}\vec{Y}$, and by Step 2, we have \[\vec{X'} =\vec{X}\geq T_{i,z_1}\vec{Y}\geq T_{i,z_{\textcolor{black}{n+1}}}\vec{Y}.\]
			\end{enumerate}
		\end{enumerate}
	\end{proof}
	
	Let \textcolor{black}{$\mathcal{C}_+$} be the class of jump rates $p(\cdot,\cdot)$ with the following properties:
	\begin{enumerate}[label=A*\arabic*]	
		\item (Positive) $p(x,y) \geq 0$, if $y>x$; otherwise, $p(x,y) = 0$,
		\item (\textcolor{black}{Finite-range}) there is an $R>0$, such that $p(x,y) = 0$, for all $x-y\leq R$, \label{asm:finite range for positive jumps}
		\item (\textcolor{black}{Radially Decreasing) for $x,y\neq 0$ and $x<y$}, $p(x,y)$ is increasing in $x$, and decreasing in $y$,\label{asm:attractive for positive jumps}
		\item \textcolor{black}{(Blockage at 0)  $p(x,y)=0$ if $x=0$ or $y=0$. \label{asm:blockage at site 0}}
	\end{enumerate}  Jump rates $p(\cdot,\cdot)$ from the class \textcolor{black}{$\mathcal{C}_+$} correspond to jumps along the positive direction. To get jumps towards both direction, \textcolor{black}{ we combine two jump rates to get} $p_c(x,y)=p_+(x,y) + p_-(y,x)$ where both $p_+, p_-$ are from the class \textcolor{black}{$\mathcal{C}_+$. We shall denote the colletcion of $p_c$ as $\mathcal{C}$.}
	
	\textcolor{black}{The main result in the following theorem is the first part, which says we can couple two AEPs with blockage $\vec{X_t}=(\vec{X}_0,\tilde{L},p,0) \succeq \vec{Y_t}=(\vec{Y}_0,\tilde{L},p_+,0)$ when they have the same jump rates $p_+$ from the class $\mathcal{C}_+$. With Remark \ref{rm:cov and reverse} and Lemma \ref{lm:combining two pairs}, we can replace $p_+$ from $\mathcal{C}_+$ by $p_c$ from $\mathcal{C}$. Lastly, we can use Lemmas \ref{lm:combining two pairs}, \ref{lm:jumps preserve ordering} to replace zero jump rates $q(\cdot)$ in either $\vec{X_t}$ or $\vec{Y_t}$ by a nonzero $q(\cdot)$.}
		
	\begin{theorem} \label{Thm: main step} Suppose jump rates $p_+(\cdot,\cdot)$, $p_-(\cdot,\cdot)$ are from the class \textcolor{black}{$\mathcal{C}_+$}.
		
		\begin{enumerate}
			\item There is a joint generator $\Omega_+$, such that for any  $\vec{X}_0 \geq \vec{Y}_0$, we can couple the pair of auxiliary processes $\vec{X_t}=(\vec{X}_0,\tilde{L},p_+,0) \succeq \vec{Y_t}=(\vec{Y}_0,\tilde{L},p_+,0)$  via $\Omega_+$.
			
			\item For combined jump rates $p_c(x,y) = p_+(x,y) + p_-(y,x)$, there is a joint generator $\Omega$, such that for any  $\vec{X}_0 \geq \vec{Y}_0$, we can couple the pair of auxiliary processes $\vec{X_t}=(\vec{X}_0,\tilde{L},p_c,0) \succeq \vec{Y_t}=(\vec{Y}_0,\tilde{L},p_c,0)$  via $\Omega$.
			
			\item (Theorem \ref{thm:coupling}) Let $q(\cdot):\mathbb{Z}\setminus\{0\}\to \mathbb{R}_{\geq 0}$, and $p_c(x,y) = p_+(x,y) + p_-(y,x)$. There are generators  $\Omega_R$, and $\Omega_L$, such that such that for any  $\vec{X}_0 \geq \vec{Y}_0$, we can couple $\vec{X_t}=(\vec{X}_0,\tilde{L}_R,p_c,q) \succeq \vec{Y_t}=(\vec{Y}_0,\textcolor{black}{\tilde{L}},p_c,0)$  via $\Omega_R$, and $\vec{W_t}=(\vec{X}_0,\tilde{L},p_c,0) \succeq \vec{Z_t}=(\vec{Y}_0,\textcolor{black}{\tilde{L}_L},p_c,q)$  via $\Omega_L.$
		\end{enumerate}
	\end{theorem}
	\begin{proof}
		\begin{enumerate}
			\item By \textcolor{black}{Lemma \ref{lm: coupling preserve i-th}},
			for any $\vec{X} \geq \vec{Y}$ and $0<z\leq R$, we can find a $z' = C(\vec{X},\vec{Y},i,R,z)\geq 0$, such that
			\[T_{i,z'}\vec{X} \geq T_{i,z}\vec{Y}.\] \textcolor{black}{ (One choice for $C(\vec{X},\vec{Y},i,R,z)$ is the function constructed inductively in the proof of Lemma \ref{lm: coupling preserve i-th}.)} Therefore, we can assign the jump rates for the i-th particles by following functions: 
			\begin{align} 
			p_{i,s,z}(\vec{X},\vec{Y}) &:= \begin{cases}
			\mathbb{1}_{A_{i,z}}(\vec{Y}) \cdot p(Y_i,Y_i+z) \label{def:p_i,s,z} &, \text{ if } \textcolor{black}{s  =  C(\vec{X},\vec{Y},i,R,z), \text{ and }\vec{X}\geq \vec{Y}},  \\
			0 &, \text{ else },
			\end{cases}\\
			p_{i,s,0}(\vec{X},\vec{Y}) &:=\begin{cases} \mathbb{1}_{A_{i,s}}(\vec{X})\left(p(X_i,X_i+s) - 
			\sum_{0<z\leq R}p_{i,s,z}(\vec{X},\vec{Y})\right) &,\text{ if } \textcolor{black}{s  >0}\\
			0 &,\text{ if } s =0 . \label{def:p_i,s,0}
			\end{cases}
			\end{align}
			\textcolor{black}{Particularly, by \eqref{def:p_i,s,z}, at most one term in the sum of \eqref{def:p_i,s,0} is positive with value $p(Y_i,Y_i+z)$ for some $z\leq R$. By Lemma \ref{lm: coupling preserve i-th}, we get $X_i\geq Y_i, X_i+s\leq Y_i+z$, which implies $p_{i,s,0} \geq 0$ by Assumption \ref{asm:attractive for positive jumps}.}  
			
			\textcolor{black}{The we define the generator $\Omega_+$ by its action on a local function $F$ by:} 
			if $\vec{X}\geq \vec{Y}$,
			\begin{align} 
			\Omega_+F(\vec{X},\vec{Y}) =& \sum_{i, 0<z\leq R,0\leq s\leq R} p_{i,s,z}(\vec{X},\vec{Y})\left[F(T_{i,s}\vec{X},T_{i,z}\vec{Y}) -F(\vec{X},\vec{Y}) \right] \label{simultaneuous jump}\\
			+& \sum_{i,0<s\leq R} p_{i,s,0}(\vec{X},\vec{Y})\left[F(T_{i,s}\vec{X},\vec{Y}) -F(\vec{X},\vec{Y}) \right],
			\label{single X jump}
			\end{align}
			\textcolor{black}{if $\vec{X}\ngeq \vec{Y},$
			\begin{align} 
			\Omega_+F(\vec{X},\vec{Y}) =& \sum_{i, 0<z\leq R} \mathbb{1}_{A_{i,z}}p(Y_i,Y_i+z)(\vec{Y})\left[F(\vec{X},T_{i,z}\vec{Y}) -F(\vec{X},\vec{Y}) \right] \label{single Y jump2}\\
			+& \sum_{i,0<s\leq R}\mathbb{1}_{A_{i,s}}(\vec{X})p(X_i,X_i+z) \left[F(T_{i,s}\vec{X},\vec{Y}) -F(\vec{X},\vec{Y}) \right]
			\label{single X jump2}
			\end{align}}
			\textcolor{black}{\eqref{simultaneuous jump}} corresponds to the case in \textcolor{black}{Lemma \ref{lm: coupling preserve i-th}} when both of the i-th particles in $\vec{X}$ and $\vec{Y}$ jump, while \eqref{single X jump} corresponds to the case where only the i-th particle in $\vec{X}$ jumps; \eqref{single Y jump2},\eqref{single X jump2} correspond to the case where particles in $\vec{X},\vec{Y}$ jump independently. The rest is to check $\Omega_+$ satisfies Definition \ref{def:coupling}. This is standard:
			
			The initial configuration can always be chosen with $\vec{W} \geq \vec{V} $ almost surely and  $\vec{W}\overset{d}{=} \vec{X}_0$, $\vec{V}\overset{d}{=} \vec{Y}_0$. (See Theorem B9\cite{Li})
			
			To show $\vec{W}_t \geq \vec{V}_t$ almost surely, use the same arguments in the proof of Lemma \ref{lm:combining two pairs}. We want to show the closed set $F_0=\{(\vec{X},\vec{Y}):\vec{X} \geq \vec{Y}\}$ is an absorbing set by checking \textcolor{black}{$\Omega_+\mathbb{1}_{F_0}(\vec{X},\vec{Y})\geq 0:$}
			\begin{enumerate}
				\item for $\vec{X}\geq \vec{Y}$, by Lemma \ref{lm: coupling preserve i-th} and $p_{i,s,z} \geq 0$
				\begin{align*}
				\Omega_+\mathbb{1}_{F_0}(\vec{X},\vec{Y}) =& \sum_{\substack{i\in \mathbb{Z}, 0<z\leq R,\\0\leq s\leq R}} p_{i,s,z}(\vec{X},\vec{Y})\left[\mathbb{1}_{F_0}(T_{i,s}\vec{X},T_{i,z}\vec{Y}) -\mathbb{1}_{F_0}(\vec{X},\vec{Y}) \right] \\
				&+ \sum_{i\in \mathbb{Z},0<s\leq R} p_{i,s,0}(\vec{X},\vec{Y})\left[\mathbb{1}_{F_0}(T_{i,s}\vec{X},\vec{Y}) -\mathbb{1}_{F_0}(\vec{X},\vec{Y}) \right]\\
				=& 0.
				\end{align*}
				
				\item for $\vec{X} \ngeq \vec{Y}$, it's obvious that $\Omega_+\mathbb{1}_{F_0}(\vec{X},\vec{Y})\geq 0$ since each term is nonnegative. We only need to show the sum is finite. Notice that only finitely many terms in \eqref{single X jump2} are positive. Since $T_{i,s}$ changes finitely many $X_i$, if one term $T_{i,s}\vec{X} \geq \vec{Y}$ holds while $\vec{X}\ngeq\vec{Y}$,  $T_{i',s'}\vec{X}\geq \vec{Y}$ holds for finitely many pairs $i',s'$. Similarly, only finitely many terms in \eqref{single Y jump2} are positive. Therefore, 	$\Omega_+\mathbb{1}_{F_0}(\vec{X},\vec{Y})\geq 0.$ 
			\end{enumerate}
			To show the marginal conditions,
			\textcolor{black}{we will check for $F_2(\vec{X},\vec{Y})=H_2(\vec{Y})$, and the other follows directly from $T_{i,0}\vec{X} = \vec{X}$, \eqref{def:p_i,s,z} and \eqref{def:p_i,s,0}. On $F_0^c = \{(\vec{X},\vec{Y}):\vec{X}\ngeq\vec{Y})\}$, clearly ${\Omega_+H_2(\vec{X},\vec{Y}) = \tilde{L}H_2(\vec{Y})}$. We only need for every $\vec{X}\geq \vec{Y}$,}
			\begin{align*}
			\Omega_+F_2(\vec{X},\vec{Y}) =& \sum_{i \in \mathbb{Z}, 0<z\leq R,0\leq s\leq R} p_{i,s,z}(\vec{X},\vec{Y})\left[H_2(T_{i,z}\vec{Y}) -H_2(\vec{Y}) \right] \\
							&+ \sum_{i \in \mathbb{Z},0<s\leq R} p_{i,s,0}(\vec{X},\vec{Y})\left[H_2(\vec{Y}) -H_2(\vec{Y}) \right]  \\
			=& \sum_{i  \in \mathbb{Z}, 0< z\leq R,0\leq s\leq R} p_{i,s,z}(\vec{X},\vec{Y})\left[H_2(T_{i,z}\vec{Y})-H_2(\vec{Y}) \right] \\
			=& \sum_{i \in \mathbb{Z}, 0< z\leq R}\left(\sum_{0\leq s\leq R} \mathbb{1}_{\{s=C(\vec{X},\vec{Y},i,R,z)\}} \right) \cdot \mathbb{1}_{A_{i,z}}(\vec{Y}) p(Y_i,Y_i+z)\left[H_2(T_{i,z}\vec{Y})-H_2(\vec{Y}) \right] \\
			=& \sum_{i \in \mathbb{Z}, 0< z\leq R} \mathbb{1}_{A_{i,z}}(\vec{Y}) p(Y_i,Y_i+z)\left[H_2(T_{i,z}\vec{Y})-H_2(\vec{Y})\right] = \tilde{L}H_2(\vec{Y}).
			\end{align*} The fourth equality is due to Lemma \ref{lm: coupling preserve i-th}, which implies exacly one $s$ satisfies $s = C(\vec{X},\vec{Y},i,R,z)$.
%

			\item 	The second part is an application of Lemma \ref{lm:combining two pairs}, the change of variable argument in Remark \ref{rm:cov and reverse} and the first part.
			
			Let \textcolor{black}{$\vec{X}_{-,t} =(R(\vec{X}_{0}),\tilde{L},\tilde{p}_-,0) $, where $\tilde{p}_-(x,y)=p_-(y,x)$}. Then, $R(\vec{X}_{-,t})= (\vec{X}_{0},\tilde{L},p_-,0)$. As $R(\cdot)$ is a map reversing ordering,
			\[\vec{X}\geq \vec{Y} \Leftrightarrow R(\vec{X}) \leq R(\vec{Y}).\]
			\textcolor{black}{By the first part of Theorem \ref{Thm: main step}, we can couple  $\vec{X}_{-,t} =(R(\vec{X}_{0}),\tilde{L},\tilde{p}_-,0)   \preceq (R(\vec{Y}_{0}),\tilde{L},\tilde{p}_-,0) =\vec{Y}_{-,t}$ for any $\vec{X}_0 \geq \vec{Y}_0$ via a generator. 
			Therefore, there is a generator $\Omega_-$, via which we can couple $R(\vec{X}_{-,t})=(\vec{X}_{0},\tilde{L},p_-,0)\succeq R(\vec{Y}_{-,t}) =(\vec{Y}_{0},\tilde{L},p_-,0)$ for any $\vec{X}_0 \geq \vec{Y}_0$.} Then by Lemma \ref{lm:combining two pairs}, we get the joint generator $\Omega=\Omega_++\Omega_-$.
			
			\item This is a consequence of the second part, Lemma \ref{lm:combining two pairs} and Lemma \ref{lm:jumps preserve ordering}. Take $\Omega_R = \Omega_{0,R} + \Omega$, and $\Omega_L = \Omega_{0,L} + \Omega$. \textcolor{black}{We will show the first case, and the other is similar:}
			
			\textcolor{black}{
			By the second part of Theorem \ref{Thm: main step}, we have a generator $\Omega$, via which we can couple auxiliary processes 
			\[\vec{X}_t= (\vec{X}_0,\tilde{L},p_c,0) \succeq (\vec{Y}_0,\tilde{L},p_c,0)=\vec{Y}_t, \] for any $\vec{X}_0\geq \vec{Y}_0$. By Lemma \ref{lm:jumps preserve ordering}, we can also find a generator $\Omega_{0,R}$ to couple auxiliary processes
			\[ \vec{W}_t= (\vec{X}_0,\tilde{L}_R,0,q) \succeq (\vec{Y}_0,\tilde{L},0,0)=\vec{Z}_t,
			\] for any $\vec{X}_0\geq \vec{Y}_0$. Notice that $\vec{X}_t$ is also $(\vec{X}_0,\tilde{L}_R,p_c,0)$. By Lemma \ref{lm:combining two pairs}, we can use generator $\Omega_R = \Omega_{0,R} + \Omega$ to couple
			\[\vec{U}_t= (\vec{X}_0,\tilde{L}_R,p_c,q) \succeq (\vec{Y}_0,\tilde{L}_R,p_c,0) =\vec{V}_t\] for any $\vec{X}_0 \geq \vec{Y}_0$.} 			
		\end{enumerate}
	\end{proof}
 
 In the proof of the first part of Theorem \ref{Thm: main step}, we see $p_{i,s,z}$ and $p_{i,s,0}$ defined by \eqref{def:p_i,s,z} and \eqref{def:p_i,s,0} are important in constructing the joint generator $\Omega_+$ defined by \eqref{single X jump}-\eqref{single X jump2}. They require Lemma \ref{lm:jumps preserve ordering} and assumption \ref{asm:attractive for positive jumps}. The Lemma \ref{lm:jumps preserve ordering} depends on the finite range $R$ and $\vec{X}\geq \vec{Y}$, while the latter is an assumption on the jump rates. We can easily modify $p_{i,s,z}$,  $p_{i,s,0}$ and $\Omega_+$ to couple two modified auxiliary processes defined in Section \ref{sec:Lower bounds for a fast tagged particle}. 
	\begin{corollary}\label{cor: coupling of maprocess} Let $p(\cdot)$ satisfy assumption \ref{Asmp:attractiveness2}, \ref{Asmp:finite-range2}, and $q(\cdot)$ be of range $R$ with an extra condition
		\begin{equation}\begin{cases} \label{eq:q,p condition}
		q(k)\geq p(k), \text{ if }k>0, \\
		q(k)\leq p(k), \text{ if }k<0,
		\end{cases} 
		\end{equation} Then we can find a joint generator $\tilde{\Omega}$ to couple modified auxiliary processes $(\vec{X}_t,I_t)=(\vec{X}_0,p,q,I_t)$ and $(\vec{Y}_t,i_t)=(\vec{Y}_0,p,p,i_t)$ for any initial condition $\vec{X}_0 \geq \vec{Y_0}$, in the sense
		\begin{align*}
			\vec{X}_t \geq \vec{Y}_t,& \text{ for all } t\geq 0, \\
		\tilde{\Omega} F_1\left(\vec{X},I,\vec{Y},i\right) =& \hat{L}_{p,q} H_1\left(\vec{X},I\right),  \\
		\tilde{\Omega} F_2\left(\vec{X},I,\vec{Y},i\right) =& \hat{L}_{p,p} H_2\left(\vec{Y},i\right), 
		\end{align*} 
			for any local functions $F_1\left(\vec{X},I,\vec{Y},i\right)=H_1\left(\vec{X},I\right)$ and $F_2\left(\vec{X},I,\vec{Y},i\right)=H_2\left(\vec{Y},i\right)$.
	\end{corollary}    
	\begin{proof}
		 We will give the joint generator $\tilde{\Omega} =\tilde{\Omega}_+ + \tilde{\Omega}_-$ by writing out $\tilde{\Omega}_+$ and $\tilde{\Omega}_-$, which will have the same form in terms of $p_{j,s,z}$. The rest is to check conditions, which follows almost the same arguments as those in the first part of Theorem \ref{Thm: main step}, and we will omit it. 
		 
		 We first define the modified $\tilde{\Omega}_+$ by modifying $p_{j,s,z}$, $p_{j,s,0}$ from \eqref{def:p_i,s,z} and \eqref{def:p_i,s,0}: 
		 
		 For $R\geq z> 0$, $R\geq s \geq 0$,
		 \begin{align} 
		 		\tilde{p}_{j,s,z}\left(\vec{X},I,\vec{Y},i\right) &:= \begin{cases}
		 		\mathbb{1}_{A_{j,z}}(\vec{Y}) \cdot p(z) \label{def:p_i,s,z mod} &, \text{ if } s  =  C(\vec{X},\vec{Y},j,R,z), \text{ and }\vec{X}\geq \vec{Y}  \\
		 		0 &, \text{ else }
		 		\end{cases}\\
		 		\tilde{p}_{j,s,0}\left(\vec{X},I,\vec{Y},i\right) &:= \begin{cases}
		 		\mathbb{1}_{A_{j,s}}(\vec{X})\left(p(s) - 
		 		\sum_{0<z\leq R}\tilde{p}_{j,s,z}\left(\vec{X},I,\vec{Y},i\right)\right) &,\text{ if } j\neq I \\
		 		\mathbb{1}_{A_{j,s}}(\vec{X})\left(q(s) - 
		 		\sum_{0<z\leq R}\tilde{p}_{I,s,z}\left(\vec{X},I,\vec{Y},i\right)\right)&,\text{ if } j= I
		 		\end{cases} \label{def:p_i,s,0 mod}	
		 \end{align} where $C(\vec{X},\vec{Y},j,R,z)$ is the function constructed in Lemma \ref{lm: coupling preserve i-th}. If we replace $q$ by $p$, \eqref{def:p_i,s,0 mod} is the same as \eqref{def:p_i,s,0}. Therefore, it is nonnegative by condition \eqref{eq:q,p condition}.
		 Then the generator $\tilde{\Omega}_+$ acts on $F$ is given by:  if $\vec{X}\geq \vec{Y}$,
		 \begin{align*} 
		 \tilde{\Omega}_+F\left(\vec{X},I,\vec{Y},i\right) =& \sum_{\substack{j \in \mathbb{Z}, 0\leq z\leq R, \\0\leq s\leq R}} \tilde{p}_{j,s,z}\left(\vec{X},I,\vec{Y},i\right)\left[F\left(T_{j,s}\vec{X},\hat{I}_{j,s}\left(\vec{X},I\right),T_{j,z}\vec{Y},\hat{I}_{j,z}\left(\vec{Y},i\right)\right) -F\left(\vec{X},I,\vec{Y},i\right) \right],
		 \end{align*}
		 and if $\vec{X}\ngeq \vec{Y},$
		 	\begin{align*} 
		 	\tilde{\Omega}_+F\left(\vec{X},I,\vec{Y},i\right) =& \sum_{j \in \mathbb{Z}, 0<z\leq R} \mathbb{1}_{A_{j,z}}(\vec{Y})\cdot p(z)\left[F\left(\vec{X},I,T_{i,z}\vec{Y},\hat{I}_{j,z}\left(\vec{Y},i\right)\right) -F\left(\vec{X},I,\vec{Y},i\right) \right] \\
		 	+& \sum_{j\neq I ,0<s\leq R}\mathbb{1}_{A_{j,s}}(\vec{X})\cdot p(s) \left[F\left(T_{i,s}\vec{X},\hat{I}_{j,s}(\vec{X},I),\vec{Y},i\right) -F\left(\vec{X},I,\vec{Y},i\right) \right]\\
		 	+& \sum_{0<s\leq R}\mathbb{1}_{A_{I,s}}(\vec{X})\cdot q(s) \left[F\left(T_{i,s}\vec{X},I_{I,s}(\vec{X}),\vec{Y},i\right) -F\left(\vec{X},I,\vec{Y},i\right) \right]
		 	\end{align*}
		 	where $\hat{I}_{j,z}\left(\vec{X},I\right)$ is defined by \eqref{eq: modified index change}, which is the same as $I_{I,z}(\vec{X})$ when $j=I$. On the other hand, we can also define $\tilde{\Omega}_-$ in a similar way:
		 	
		 	For $R\geq -s> 0$, $R\geq -z \geq 0$,
		 	\begin{align} 
		 	\tilde{p}_{j,s,z}\left(\vec{X},I,\vec{Y},i\right) &:= \begin{cases}
		 	\mathbb{1}_{A_{i,z}}(\vec{Y}) \cdot q(s)  &, \text{ if } j=i, z  =  -C(R(\vec{Y}),R(\vec{X}),-i,R,-s), \text{ and }\vec{X}\geq \vec{Y}  \\
		 	\mathbb{1}_{A_{j,z}}(\vec{Y}) \cdot p(s)  &, \text{ if } j\neq i, z  =  -C(R(\vec{Y}),R(\vec{X}),-j,R,-s), \text{ and }\vec{X}\geq \vec{Y}  \\
		 	0 &, \text{ else }
		 	\end{cases}\label{def:p_i,s,z mod-} \\
		 	\tilde{p}_{j,0,z}\left(\vec{X},I,\vec{Y},i\right) &:=
		 	\mathbb{1}_{A_{j,z}}(\vec{X})\left(p(z) - 
		 	\sum_{0<-s\leq R}\tilde{p}_{j,s,z}\left(\vec{X},I,\vec{Y},i\right)\right)  \label{def:p_i,s,0 mod-}
		 	\end{align}	where $R(\vec{X})$ is defined in Remark \ref{rm:cov and reverse}. Also, by replacing $q$ by $p$ in \eqref{def:p_i,s,z mod-} and using condition \eqref{eq:q,p condition}, we see both \eqref{def:p_i,s,z mod-} and \eqref{def:p_i,s,0 mod-} are nonnegative. The generator $\tilde{\Omega}_-$ acts on $F$ is given by:  if $\vec{X}\geq \vec{Y}$,
		 	\begin{align*} 
		 	\tilde{\Omega}_-F\left(\vec{X},I,\vec{Y},i\right) =& \sum_{\substack{j, 0\leq -s\leq R, \\0\leq -z\leq R}} \tilde{p}_{j,s,z}\left(\vec{X},I,\vec{Y},i\right)\left[F\left(T_{j,s}\vec{X},\hat{I}_{j,s}\left(\vec{X},I\right),T_{j,z}\vec{Y},\hat{I}_{j,z}\left(\vec{Y},i\right)\right) -F\left(\vec{X},I,\vec{Y},i\right) \right],
		 	\end{align*}
		 	and if $\vec{X}\ngeq \vec{Y},$
		 	\begin{align*} 
		 	 	\tilde{\Omega}_-F\left(\vec{X},I,\vec{Y},i\right) =& \sum_{j \in \mathbb{Z}, 0<-z\leq R} \mathbb{1}_{A_{j,z}}(\vec{Y})\cdot p(z)\left[F\left(\vec{X},I,T_{i,z}\vec{Y},\hat{I}_{j,z}\left(\vec{Y},i\right)\right) -F\left(\vec{X},I,\vec{Y},i\right) \right] \\
		 	 	+& \sum_{j\neq I ,0<-s\leq R}\mathbb{1}_{A_{j,s}}(\vec{X})\cdot p(s) \left[F\left(T_{i,s}\vec{X},\hat{I}_{j,s}(\vec{X},I),\vec{Y},i\right) -F\left(\vec{X},I,\vec{Y},i\right) \right]\\
		 	 	+&  \sum_{0<-s\leq R}\mathbb{1}_{A_{I,s}}(\vec{X})\cdot q(s) \left[F\left(T_{I,s}\vec{X},I_{I,s}(\vec{X}),\vec{Y},i\right) -F\left(\vec{X},I,\vec{Y},i\right) \right]
		 	 	\end{align*}
		 	 	We can see both $\tilde{\Omega}_+$ and $\tilde{\Omega}_-$ have the same form in terms of $\tilde{p}_{j,s,z}$. Then, we obtain $\tilde{\Omega}$ by $\tilde{\Omega} = \tilde{\Omega}_+ + \tilde{\Omega}_-$.
\end{proof}

	\end{document}